\documentclass[graybox, envcountsame, envcountsect]{svmult-fgs}

\usepackage{newtxtext,newtxmath}  
\usepackage{helvet}               
\usepackage{courier}              

\usepackage{graphicx}        
                             
\usepackage{overpic}
\usepackage[bottom]{footmisc}
\spnewtheorem*{bla}{Question}{\bf}{}
\spnewtheorem*{proofof}{\nocaption}{}{}

\newcommand{\mb}{\mathbb}
\newcommand{\mc}{\mathcal}
\newcommand{\be}{e}

\newcommand{\N}{\mathbb{N}}
\newcommand{\T}{\mathcal{T}}
\newcommand{\Te}{\mathbb{T}}
\newcommand{\C}{\mathbb{C}}
\newcommand{\R}{\mathbb{R}}
\newcommand{\e}{\epsilon}
\newcommand{\eps}{\epsilon}
\newcommand{\ga}{\gamma}

\newcommand{\sub}{\subseteq}
\newcommand{\ra}{\rightarrow}
\newcommand{\diam}{\operatorname{diam}}
\newcommand{\dist}{\operatorname{dist}}
\newcommand{\length}{\operatorname{length}}
\newcommand{\sgn}{\operatorname{sgn}}

\newcommand{\ba}{\backslash}

\renewcommand{\:}{\colon}
\newcommand{\arb}{arbitrary}
\newcommand{\arbly}{arbitrarily}

\makeindex             

\begin{document}

\title*{The continuum self-similar tree}
\author{Mario Bonk and Huy Tran}
\institute{Mario Bonk \at Department of Mathematics,
University of California, Los Angeles, CA 90095, USA, \email{mbonk@math.ucla.edu}
\and Huy Tran \at Institut f\"ur Mathematik, Technische Universit\"at Berlin,
Sekr.\ MA 7-1,
Strasse des 17.\ Juni 136,
10623 Berlin, Germany, \email{tran@math.tu-berlin.de}}

%
%
\maketitle

\abstract{We introduce the continuum self-similar tree (CSST) as the attractor of an iterated function system in the complex plane. We provide a topological characterization of the CSST and use this to relate the CSST to other metric trees such as the  continuum random tree (CRT) and Julia sets of postcritically-finite polynomials. 
 \\
\subclassname{Primary: 37C70; Secondary: 37B45.}
}

\section{Introduction}
\label{sec:1}
In this expository paper, we study the topological properties of a 
certain subset $\Te$ of the complex plane $\C$. It  is defined as the attractor of an iterated function system. As we will see, $\Te$ has a self-similar ``tree-like" structure with very regular branching behavior. In a sense it is the simplest object of this type. Sets homeomorphic to $\Te$   appear in various other contexts. Accordingly, we give the set $\Te$ a special name, and call it the {\em continuum self-similar tree} (CSST). 

To give the precise definition of $\Te$ we consider the following contracting  homeomorphisms  on 
$\C$: 
\begin{equation}\label{maps}
f_1(z) = \tfrac{1}{2}z-\tfrac{1}{2}, \quad  f_2(z) =  \tfrac{1}{2} \bar{z} + \tfrac{1}{2},  \quad f_3(z) =  \tfrac{i}{2} \bar{z} + \tfrac{i}{2}.
\end{equation}
Then the following statement is true.

\begin{proposition} \label{CSSTex}There exists a  unique 
non-empty compact set $\mb{T}\sub\C$  satisfying 
\begin{equation}\label{Tinv}
\mb{T}=f_1(\mb{T})\cup f_2(\mb{T})\cup f_3(\mb{T}).
\end{equation}
\end{proposition} 

Based on this fact, we make the following definition. 

\begin{definition}\label{def:CSST}
The continuum self-similar tree (CSST) is the set $\mb{T}\sub \C$ 
  as given by Proposition~\ref{CSSTex}.\end{definition}

  In other words, $\mb{T}$ is the attractor of the iterated function system $\{f_1,f_2,f_3\}$ in the plane.
 Proposition~\ref{CSSTex} is a special case of 
 well-known more general results in the literature (see 
\cite{Hu81}, \cite[Theorem 9.1]{Fa03}, or 
 \cite[Theorem~1.1.4]{Kigami}, for example). We will recall the argument  that leads to Proposition~\ref{CSSTex} in Section~\ref{sec:toptree}. 
 
 Spaces of a similar topological type as $\Te$  have appeared in the literature before (among the more recent examples is  the {\em antenna set} in \cite{BT01} or  
{\em Hata's tree-like set} considered in \cite[Example~1.2.9] {Kigami}).  For a representation of $\Te$ see Figure~\ref{fig:CSST}.

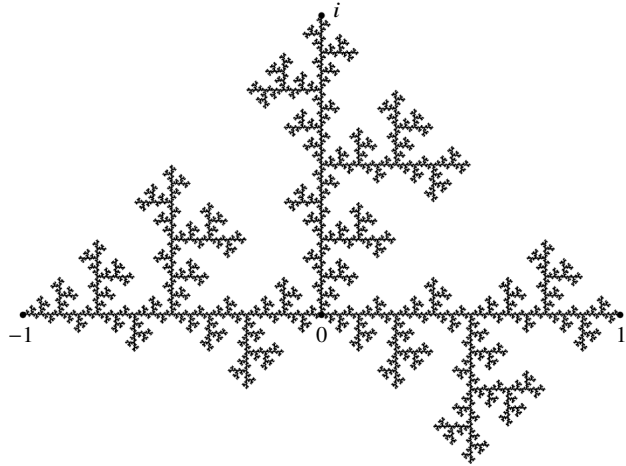
\begin{figure}
\vspace{-0.7cm}
 \begin{overpic}[ scale=0.7
    ]{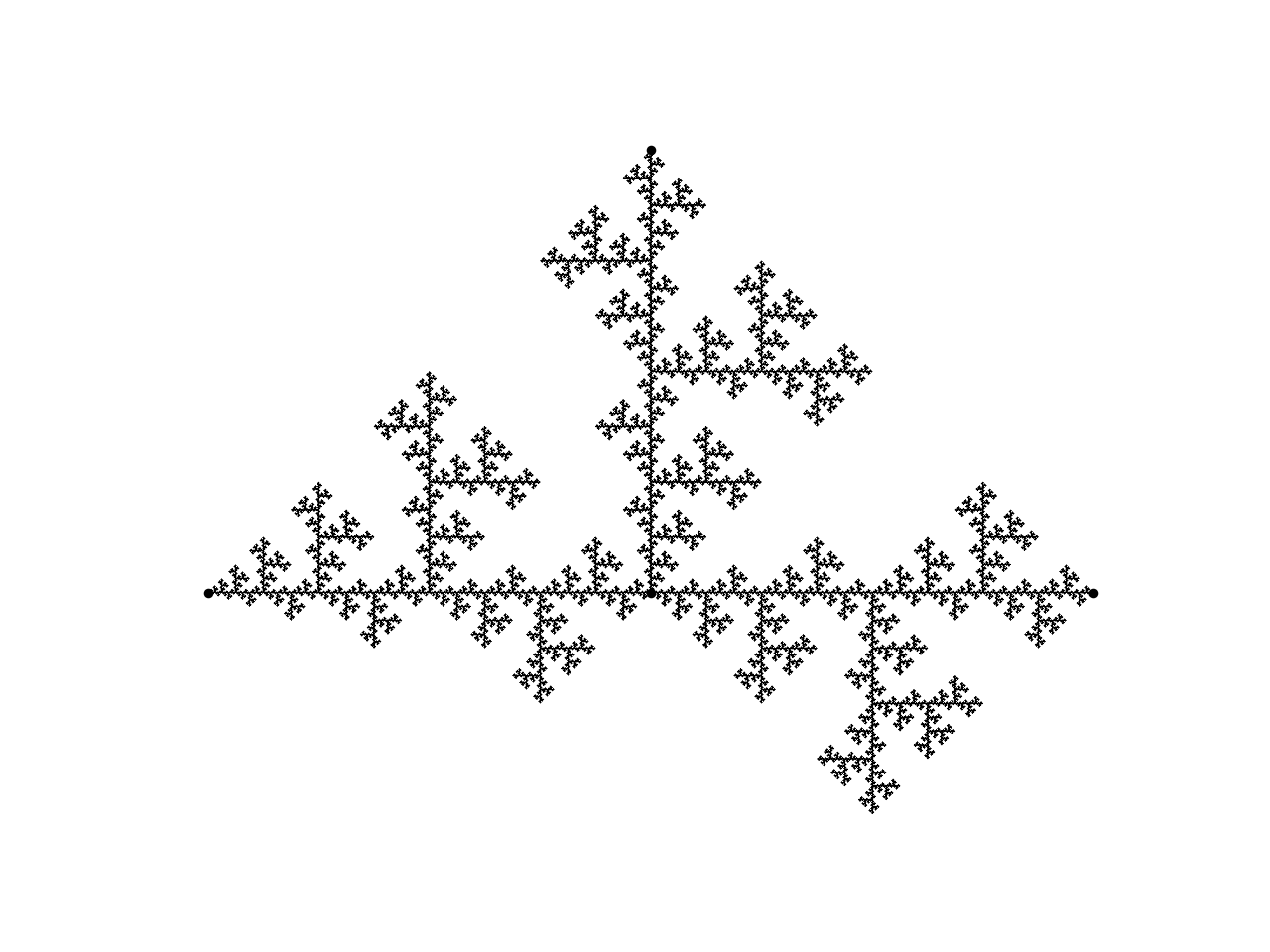}
    \put(52,63){ $i$}
    \put(14,25){ $-1$}
     \put(85,25){ $1$}
       \put(50,25){ $0$}
           \end{overpic}
  \vspace{-1cm}
\caption{The continuum self-similar tree $\Te$.}
\label{fig:CSST}
\end{figure}

To describe the topological properties of $\Te$, we introduce the  following concept.

\begin{definition}\label{def:tree} 
A {\em (metric)  tree} is a compact, connected, and locally connected  metric space $(T,d)$ containing at least two points such that for all  $a,b\in T$
with  $a\ne b$  there exists a unique arc $\alpha\sub T$ with endpoints $a$ and $b$.\end{definition}

In other words, any two distinct points $a$ and $b$ in a metric tree  can be joined by a  unique  arc $\alpha$ in $T$. It is convenient to allow $a=b$ here in which case $\alpha=\{a\}=\{b\}$ and we   consider $\alpha$ as  a {\em degenerate} arc.

In the following, we will usually call 
a metric space 
as in Definition~\ref{def:tree} a {\em tree} and drop the word ``metric" for simplicity.  It is easy to see that the concept of a tree 
 is essentially the same as   the concept of a {\em dendrite} 
that  appears in the literature (see, for example, \cite[Chapter V]{Wh}, \cite[Section \S 51 VI]{Ku68},  \cite[Chapter~X]{Na}).  More precisely, a metric space $T$ is a tree if and only if it is a non-degenerate dendrite (the simple proof is recorded in \cite[Proposition 2.2]{BM19a}).  If one drops the compactness assumption  in Definition~\ref{def:tree}, but requires in addition   
that the space is geodesic (see below for the definition), then one is 
led to the notion of a {\em real tree}. They  appear in many areas of mathematics (see \cite{LeGall, Best}, for example). 

 The following statement
 is  suggested by Figure~\ref{fig:CSST}.
 \begin{proposition}\label{prop:CSST} 
  The continuum self-similar tree   $\Te$ is a metric tree. 
  \end{proposition}


  If $T$ is a tree,  then for $x\in T$ we  denote by  $\nu_T(x)\in \N\cup\{\infty\}$   the number of (connected) components of $T\backslash \{x\}$. This number $\nu_T(x)$ is called the {\em valence}  of $x$. If  $\nu_T(x) = 1$, then $x$ is called a {\em leaf} of $T$. If $\nu_T(x)\geq 3$, then $x$ is a {\em branch point} of $T$. If  $\nu_T(x)=3$, then we also call $x$  a {\em triple point}.

The following statement is again suggested by Figure~\ref{fig:CSST}.
\begin{proposition}\label{prop:T123}
Each branch point of the tree $\Te$ is a triple point, and these triple points are dense in $\Te$. 
\end{proposition}

The set $\Te$  has an interesting geometric property,  namely it is a {\em quasi-convex} subset of $\C$., i.e., any two points 
in $ \Te$ can be joined by a path whose length is comparable to the distance of the points.

\begin{proposition} \label{prop:quasicon} There exists a constant 
$L>0$ with the following property: if $a,b\in \Te$ and $\alpha$ is the unique arc in $\Te$ joining  $a$ and $b$, then 
$$ \length (\alpha)\le L |a-b|. $$  
\end{proposition}

Note that a unique (possibly degenerate) arc $\alpha\sub \Te$ joining $a$ and $b$ exists, because $\Te$ is a 
tree  according to 
 by Proposition~\ref{prop:CSST}.  
 
 Proposition~\ref{prop:quasicon} implies that  we can define a new metric $\varrho$ on $\Te$ 
by setting 
$ \varrho(a,b)=\length(\alpha) $
for $a,b\in \Te$,  
where $\alpha$ is the unique arc in $\Te$ joining  $a$ and $b$. 
Then the metric space $(\Te, \varrho)$ is   {\em geodesic}, i.e., any two points in $(\Te, \varrho)$ can be joined by a path in $\Te$ whose length is equal to the distance of the points.  It immediately follows from  Proposition~\ref{prop:quasicon} that metric spaces $\Te$ (as equipped with the Euclidean metric) and $(\Te, \varrho)$ are bi-Lipschitz equivalent by the identity map.

A natural  way to construct $(\Te,\varrho)$, at least  as an abstract metric space, is as follows. We start with a  line segment $J_0$ of length $2$. Its  midpoint $c$ subdivides  $J_0$ into two line segments of length $1$. We glue 
to $c$ one of the endpoints of another line segment $s$ of the same length. Then we obtain a set  $J_1$ consisting of three line segment of length $1$. The set $J_1$ carries the natural  path metric. We now repeat this procedure inductively. At the $n$th step we obtain a tree $J_n$ consisting of 
$3^n$ line segments of length $2^{1-n}$. To pass to $J_{n+1}$, each of these line segments $s$ is subdivided by its midpoint $c_s$ into two line segment of length $2^{-n}$ and we glue to $c_s$ one endpoint of another line segment of length $2^{-n}$.  

In this way, we obtain  an ascending sequence $J_0\sub J_1\sub \ldots $  of  trees equipped with  a geodesic metric. 
The union $J=\bigcup_{n\in \N_0}J_n $ carries a natural path metric $\varrho$ that agrees with the metric on $J_n$ for each $n\in \N_0$. As an abstract space  one can define $(\Te,\varrho)$ as the completion of the metric 
space $(J,\varrho)$.

 If one wants to realize $\Te$ as a subset of $\C$ by this construction, one starts with the initial line segment
$J_0=[-1,1]$, and adds $s=[0,i]$ in the first step
to obtain $J_1=[-1,0]\cup [0,1]\cup [0,i]$. Now one wants to choose  suitable Euclidean similarities $f_1$, $f_2$, $f_3$ that copy the interval $[-1,1]$ to $[-1,0]$, $[0,1]$, $[0,i]$, respectively.
One  hopes to  realize $J_n$  as a subset of $\C$ using an inductive procedure based on 
$$J_{n+1}=f_1(J_n)\cup f_2(J_n)\cup f_3(J_n), \quad n\in \N_0. $$
In order to  avoid self-intersections and   ensure that each set $J_n$ is indeed a tree, one has to be  careful about the orientations of the maps $f_1$, $f_2$, $f_3$. The somewhat non-obvious choice  of these  maps as in  \eqref{maps} leads to the desired result.  See Proposition~\ref{prop:CSSTgeo}  and the discussion near the end of Section~\ref{s:CSST} 
for a precise statements how to use the maps in \eqref{maps} to realize the sets $J_n$ as subsets of $\C$, and obtain $\Te$ (as in Definition~\ref{def:CSST}) as the closure 
of $\bigcup_{n\in \N_0} J_n$. A representation of $J_5$ is shown in Figure~\ref{fig:J5}.

\begin{figure}
\vspace{-0.7cm}
 \begin{overpic}[ scale=0.7
    ]{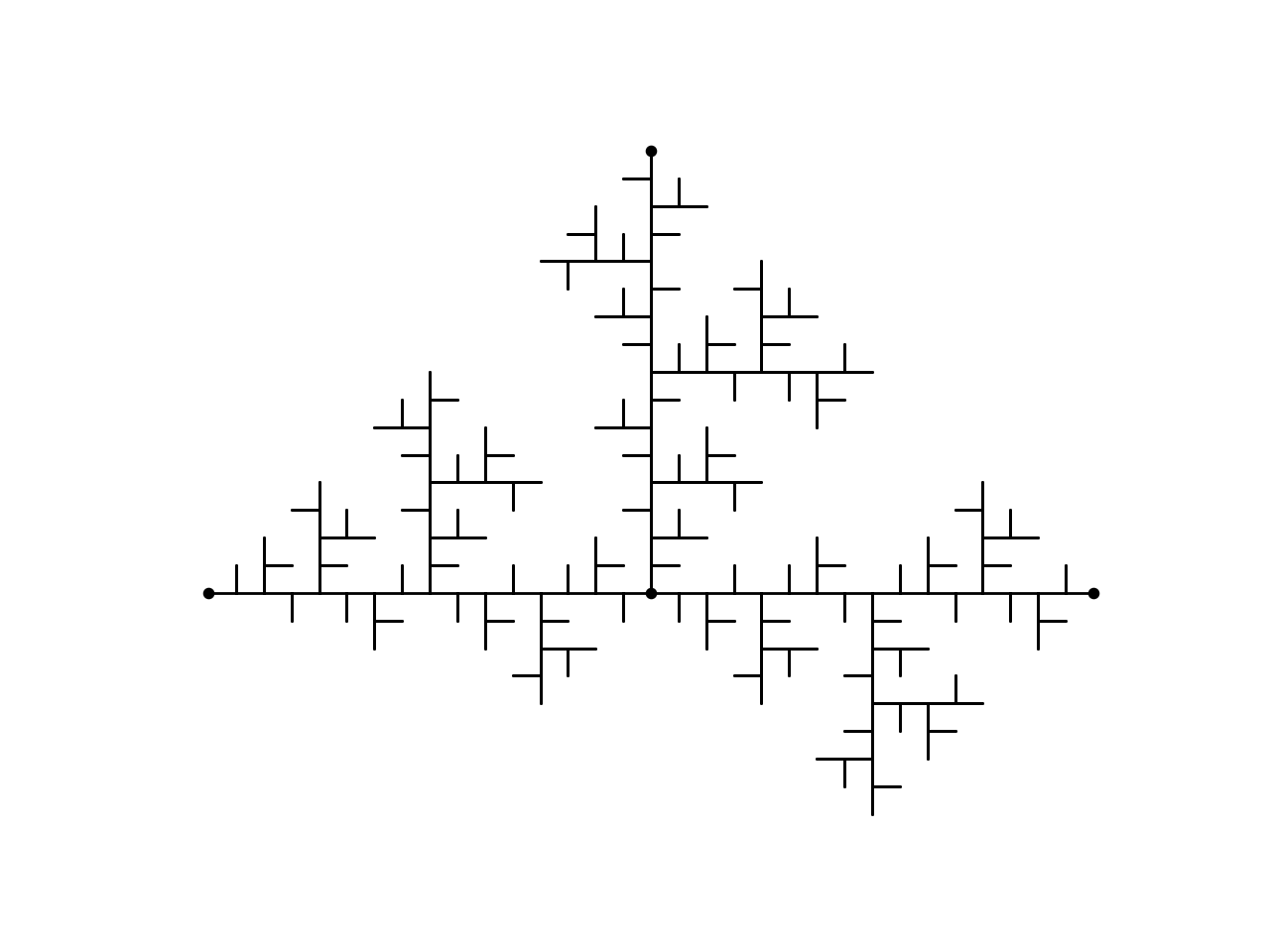}
    \put(52,63){ $i$}
    \put(14,25){ $-1$}
     \put(85,25){ $1$}
       \put(50,25){ $0$}
           \end{overpic}
  \vspace{-1cm}
\caption{The set $J_5$.}
\label{fig:J5}
\end{figure}
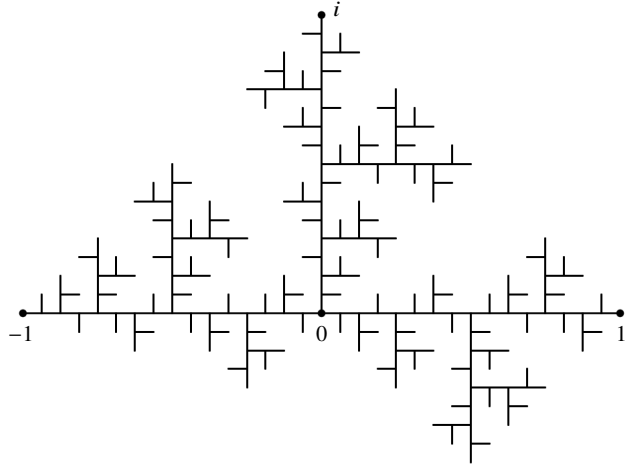

%
%

The conditions in Proposition~\ref{prop:T123} actually characterize the CSST topologically.

\begin{theorem}
\label{criterion}
A metric tree $(T,d)$ is homeomorphic to the continuum self-similar tree  $\Te$  if and only if the following conditions are true:
\begin{itemize}
\item[\textnormal{(i)}] For every point $x\in T$  we  have 
$\nu_T(x)\in\{1,2,3\}$. 

\smallskip 
\item[\textnormal{(ii)}] \ The set of triple points $\{x\in T: \nu_T(x) = 3\}$ is a   dense subset of $T$.
\end{itemize}
\end{theorem}

We will  derive Theorem \ref{criterion}  from a slightly more general statement.  For its formulation let 
 $m\in \mb{N}$ with $m\geq 3$. We consider the class $\mc{T}_m$ consisting of all metric trees $T$ such that
 \begin{itemize}
\item[{(i)}] for every point $x\in T$  we  have 
$\nu_T(x)\in \{1,2,m\}$, and 

\smallskip 
\item[(ii)] \ the set of branch  points $\{x\in T: \nu_T(x)=m \}$ is a   dense subset of $T$.
\end{itemize}
Note that by Proposition~\ref{prop:T123} the CSST  $\Te$ satisfies the conditions in Theorem~\ref{criterion} with $m=3$, and so
$\Te$ belongs to the class of trees $\T_{3}$. Now the following statement is true which contains Theorem~\ref{criterion} as a special case. 

 \begin{theorem}\label{t:infinite} Let $m\in \N$ with $m\ge 3$.  Then all trees in $\mc{T}_{m}$ are homeomorphic to each other.
\end{theorem}

 Theorems~\ref{criterion} and \ref{t:infinite}  are not new. In a previous version of this paper, we considered Theorem~\ref{criterion} as a ``folklore"  statement,
  but we did not have a reference for a proof. Later, the paper 
\cite{CG94} was brought to our attention which contains a more general result which implies Theorem~\ref{t:infinite}, and hence also Theorem~\ref{criterion}  (see \cite[Theorem 6.2]{CG94}; the proof there seems to be incomplete though---the continuity of the map $h$ on the dense subset of $X$ needs more justification). 
Theorem~\ref{t:infinite} was explicitly stated in \cite[(6), p.~490]{Ch80}, but it  seems that the origins of Theorem~\ref{t:infinite} can be traced back much further to \cite{Wa23} (see also \cite[Chapter X]{Me32}, and \cite{CC98} for more pointers to the relevant older literature about dendrites). 

We will give a complete proof of Theorem~\ref{t:infinite}. It  is based on  ideas that are quite different from those in \cite{CG94}, but we consider our method of proof  very natural.  It is also related to some other recent work, in particular \cite{BM19a, BM19b};  so one can view the present paper as an introduction to these ideas.  We will say more about our motivation below.

Our proof of Theorem~\ref{t:infinite} can be outlined as follows. 
 Fix $m$ as in the statement and consider a tree $T$ in $\T_m$. 
 Then we cut $T$ into  $m$
 subtrees 
 at  a carefully chosen branch point. This process is repeated inductively. One labels the subtrees obtained in this way by finite 
 words consisting of  letters in the   alphabet   $\mathcal{A} =\{1,2, \dots, m\}$. The labels are  chosen so that  if $S$ is another tree in $\T_m$ and one decomposes $S$ in a similar manner, then  one has the same combinatorics (i.e., intersection and inclusion pattern)  for the subtrees in $T$ and $S$. The desired homeomorphism 
 between  $T$ and $S$ can then be obtained from a general statement that produces a homeomorphism between two spaces, if they admit matching decompositions  into pieces satisfying suitable conditions
 (see  Proposition~\ref{nested}).

The CSST is related to metric trees appearing in other areas of mathematics. One of these objects is the 
 {\em (Brownian) continuum random tree} (CRT). This is a random  tree introduced by  Aldous \cite{AldousI} when he studied the scaling limits of simplicial trees arising  from the critical Galton-Watson process. One can describe the  CRT as follows. We consider a sample of Brownian excursion $(\be_t)_{0\leq t\leq 1} $ on the interval $[0,1]$. For $s,t\in [0,1]$, we set 
$$d_{\be}(s,t) = \be(s) + \be(t) - 2 \inf\{ \be(r): \min(s,t)\le r \le 
\max(s,t)\}.$$ Then $d_{\be}$ is a pseudo-metric on $[0,1]$. 
We define an equivalence relation on $[0,1]$ by setting $s\sim t$ if $d_{\be}(s,t) = 0$. Then $d_{\be}$ descends to a metric on  the quotient space $T_{\be}=[0,1]/\sim$. The metric space 
$(T_{\be}, d_{\be})$
  is almost surely a metric tree (see \cite[Sections~2 and 3]{LeGall}).
Curien \cite{Curien} asked the following question.

\begin{bla}
Is the topology of the CRT almost surely constant, that is, are  two independent samples of the CRT  almost surely homeomorphic?
\end{bla}

This question was the original motivation  for the present work and  
we found  a positive answer based on the following statement.  

\begin{corollary}\label{cor:CRT} 
A sample $T$ of the  CRT is almost surely homeomorphic to the CSST~$\Te$.
\end{corollary}

\begin{proof}
As we discussed, a sample $T$ of the CRT is almost surely a 
metric tree (see \cite[Sections~2 and 3]{LeGall}).
Moreover, for such a sample $T$ 
almost surely for every point $x\in T$, the valence $\nu_T(x)$  is either 
$1$, $2$ or $3$, and the set $\{x: \nu_T(x)=3\}$ of triple points is  dense in $T$ 
(see \cite[Theorem 4.6]{Duq-LeGall} or \cite[Proposition 5.2 (i)]{LeGall}). 
It follows from Proposition~\ref{prop:T123} and 
Theorem~\ref{criterion} that a sample $T$ of the CRT is almost surely homeomorphic to the CSST $\Te$. \qed
\end{proof}

Informally, Corollary~\ref{cor:CRT}  says  that  the topology of the CRT is (almost surely) constant and given by the topology of a deterministic model space, namely the CSST. In particular, almost surely any two independent samples of the CRT are 
homeomorphic. This  answers Curien's question in the positive.   As we found out after we had obtained   proofs for Theorem~\ref{criterion} and Corollary~\ref{cor:CRT}, Curien's question had already been answered implicitly in \cite{Croydon-Hambly}. There the authors used the distributional self-similarity property of the CRT and showed that the CRT is isometric to a metric space with a random metric. This space is constructed similarly to the CSST as the attractor of an iterated 
function system with maps  very similar to \eqref{maps} (they contain an 
additional parameter though  which is unnecessary if  one uses the maps in  \eqref{maps}).




\begin{figure}
\begin{center}
\includegraphics[scale=0.7, clip=true, trim=0mm 23mm 0mm 23mm]{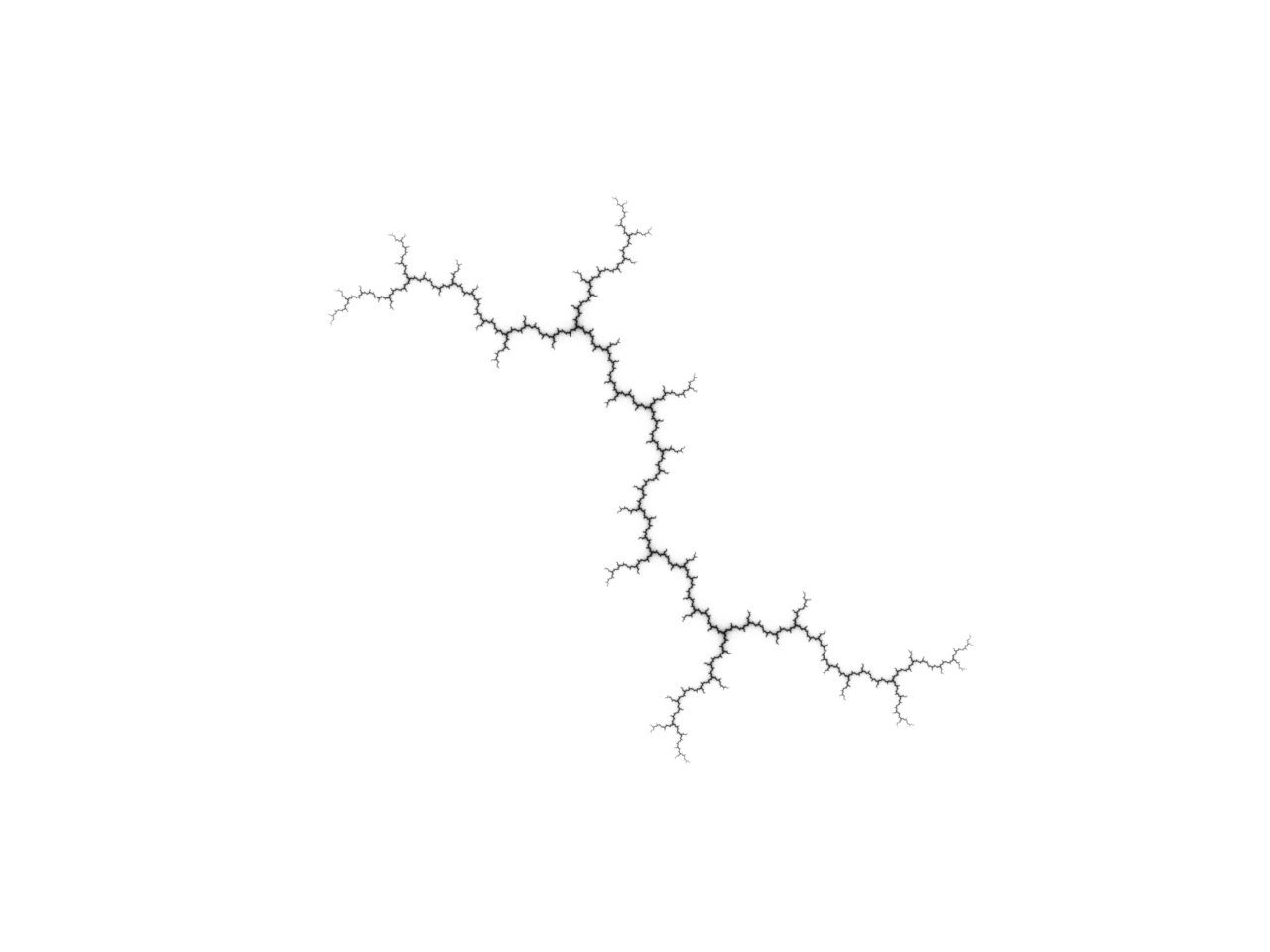}
\end{center}
\caption{The Julia set of $P(z)=z^2+i$.}
\label{fig:Julia}
\end{figure}

An important source of trees is given by  Julia sets of 
postcritically-finite polynomials without  periodic critical points in $\C$.  It follows from \cite{DH} (or see \cite[Theorem V.4.2]{CG}) that the Julia sets of such  polynomials are indeed trees. One can show    that the Julia set $\mathcal{J}(P)$ of the polynomial $P(z)=z^2+i$ (see Figure~\ref{fig:Julia}) satisfies the conditions in Theorem \ref{criterion}. Accordingly, $\mathcal{J}(P)$ is homeomorphic to 
the CSST.

%
%

%

There are several directions in which one can pursue these topics  further. For example, one can  study the topology of more general trees than those in the classes $\mc{T}_m$.  One  may want to  replace $m$ with  any finite (or even infinite) list  of allowed valences for branch points, including branch points of infinite valence. In an earlier version of our paper, we discussed 
this  in more detail. Since we learned  that 
these results are already contained in \cite{CC98}, we decided 
to skip this in the present version.  

There is one important variant 
of Theorem~\ref{t:infinite}  that we like to mention though. Namely, one can consider the (non-empty) class $\mc{T}_\infty$ of trees $T$ such that $\nu_T(x)\in \{1,2,\infty\}$ for all $x\in T$ and such that the set of branch points of $T$ (i.e., in this case the set 
$\{x\in T:\nu_T(x)=\infty\}$) is dense in  $T$. Then all trees in 
$\mc{T}_\infty$ are homeomorphic to each other (our method of proof  does not directly apply here, but 
one can  use our approach based on a more general version of Proposition~\ref{nested}). Moreover, each tree $T$ in $\mc{T}_\infty$ is {\em universal} in the sense that every tree $S$ admits a topological embedding into $T$. These results are due to Waszewski \cite{Wa23} (see \cite[Section 10.4]{Na} for a modern exposition of this universality property; see also \cite{Ch80} for a discussion of a universality property of the trees in $\mc{T}_m$, 
$m\in \N$, $m\ge 3$).

 Trees in $\mc{T}_\infty$ are also interesting, because they naturally arise in probabilistic models. More specifically, the so-called stable trees with index $\alpha\in (1,2]$  are generalizations of the CRT  (see \cite[Section~4]{LeGall} for the definition). For  fixed $\alpha \in (1,2)$, a  sample $T$ of such a stable tree  belongs to $\mc{T}_{\infty}$ almost surely  \cite[Proposition 5.2~(ii)]{LeGall}. By the previous discussion this implies  that   two independent samples of stable trees for given   $\alpha \in (1,2)$ are almost surely homeomorphic.  Note that  the Julia set of a polynomial never belongs to  $\mc{T}_{\infty}$. This follows from results due to Kiwi (see  \cite[Theorem 1.1]{Kiw02}).
 
 Another direction for further investigations are questions that are more related to  {\em geometric} properties of metric trees, in contrast to purely {\em topological} properties. In particular, one can study the {\em quasiconformal geometry} of the CSST and other trees (for a survey on the general topic of quasiconformal geometry see \cite{Bo06}).

 One of the basic notion here is the concept of quasisymmetric equivalence. By definition two metric spaces $X$ and $Y$ are called {\em quasisymmetrically equivalent} if there exists a quasisymmetry $f\: X\ra Y$. Roughly speaking,  a quasisymmetry is a homeomorphism with good geometric control: it sends metric balls to ``roundish" sets with uniformly controlled eccentricity  (for the precise definition of a quasisymmetry and other basic concepts of quasiconformal geometry see \cite{He}). 
Since every quasisymmetry is a homeomorphism, two spaces are homeomorphic if they are quasisymmetrically equivalent.
So this gives a stronger type of equivalence for metric spaces that has a more geometric flavor and goes beyond mere topology. 

A natural problem  in this context is to characterize the CCST $\Te$ up to quasisymmetric equivalence, similar to 
Proposition~\ref{prop:T123} which gives a topological characterization. This problem is solved in \cite{BM19b}. 
The precise statement is too technical to be included here, but 
roughly speaking the conditions on a metric tree $T$ to be quasisymmetrically equivalent to $\Te$ are similar in sprit to 
the conditions in Proposition~\ref{prop:T123}, but of a more ``quantitative" nature.

For example, one of the conditions stipulates that 
$T$ be {\em trivalent} (i.e., all branch points of $T$ are triple points),
but not only should the  branch points of $T$ form a dense subset of $T$,
 but  $T$ should be {\em uniformly branching} in the sense that every arc $\alpha \sub T$ contains a branch point $p$ of 
 {\em  height} $H_T(p)$ 
 comparable to the diameter of $\alpha$. Here the height 
 $H_T(p)$ is the diameter of the third largest branch of $p$
 (see the discussion around \eqref{eq:height} for more details).  
 
In our proof of  Theorem~\ref{criterion}   we 
first realized that this concept of height of a branch point plays a very important role in understanding the geometry and topology of trees. 
This concept  is also used in  \cite{BM19a, BM19b}.

The present  paper and   \cite{BM19a, BM19b} have another 
 common feature. In all of these works  it is important to have   good decompositions of the spaces studied,  depending on the problem under consideration. This line of thought in the context of quasiconformal geometry can be traced back to 
 \cite[Proposition 18.8]{BM}. More recently, Kigami \cite{Kig18} 
 has systematically investigated such decompositions 
 in the general framework  of partitions of a space given by sets  that are labeled by the vertices of a (simplicial) tree. 
This common philosophy with other recent work is the main motivation why we wanted to present the proof of  the known 
Theorem~\ref{t:infinite} from our perspective. 
 
 One can use the characterization of the CSST up to quasisymmetric equivalence established in \cite{BM19a} to prove the following statement (unpublished work by the authors): if the Julia set $\mathcal{J}(P)$ of a postcritically-finite polynomial $P$ with no  periodic critical points in $\C$ is homeomorphic to the CSST, then  $\mathcal{J}(P)$ is  quasisymmetrically equivalent to the CSST.

Finally,  we mention in passing that the geometric properties of the continuum random tree (CRT)  were considered in the recent paper \cite{LR19} by Lin and Rohde. Though
Lin and Rohde 
do not study quasisymmetric equivalence, many  of their considerations still fit   into the general framework of quasiconformal geometry.



The present  paper is organized as follows.  In  Section \ref{s:homeo}  we state and prove a  general criterion for two metric spaces to be homeomorphic based on the existence of combinatorially equivalent decompositions of the spaces. 
In Section~\ref{sec:toptree} we collect some general facts about trees that we use later. The  CSST  is  studied 
in Section~\ref{s:CSST}. There we provide proofs of 
Propositions~\ref{CSSTex},~\ref{prop:CSST},~\ref{prop:T123}, 
and~\ref{prop:quasicon}.   In Section \ref{s:proof}
 we explain how to decompose  trees in $\mc{T}_m$ with 
 $m\in \N$, $m\ge 3$. 
 Based on this, we then present  a proof Theorem~\ref{t:infinite}.   Theorem~\ref{criterion} is an immediate consequence.
 
  \section{Constructing  homeomorphisms between spaces} \label{s:homeo}
 
Throughout this paper, we use fairly standard metric space notation. If $(X,d)$ is a metric space, then we denote by $B(a,r)=\{x\in X\: d(a,x)<r\}$ the open ball of radius $r>0$ centered at $a\in X$.  
If $A,B\sub X$, then  $\diam(A)=\sup\{d(x,y):x,y\in A\}$  is the diameter of 
$A$ and $\dist(A,B)=\inf\{d(x,y): x\in A,\, y\in B\}$ the (minimal) distance of  $A$ and $B$. Similarly, if $a\in X$, then $\dist(a,B)=
\dist(\{a\},B)$ denotes the distance of the point $a$ to the set $B$. Finally, if $\ga$ is a path in $X$, then 
$\length(\ga)$ stands for  its length. 

Before we discuss trees in more detail and turn our attention to  the CSST, we will establish the following proposition that is  the key to showing that  two trees are homeomorphic. The statement will also give us some guidance for  the desired properties of tree decompositions that we will discuss in the following sections. 
 The proposition  is inspired by  \cite[Proposition 18.8]{BM}, which provided geometric conditions for the  decomposition of a space that can be used to construct quasisymmetric homeomorphisms.

\begin{proposition} \label{nested} Let $(X,d_X)$ and $(Y,d_Y)$ be  compact metric spaces. Suppose that for each $n\in \N$, the space  $X$ admits a decomposition   $X=\bigcup_{i=1}^{M_n} X_{n,i}$ as a finite union of non-empty  compact  subsets $X_{n,i}$, $i=1, \dots, M_n\in \N$, with the following 
properties for all $n$, $i$, and $j$: 
\begin{itemize}
\item[\textnormal{(i)}] Each set $X_{n+1,j}$ is the subset of some set $X_{n, i}$.

\smallskip  
\item[\textnormal{(ii)}] \ Each set $X_{n, i}$ is equal to  the union of some of the sets  $X_{n+1,j}$.

\smallskip 
\item[\textnormal{(iii)}] \ \ $\max_{1\leq i\leq M_n} \diam (X_{n, i}) \to 0$ as $n\to \infty$.
\end{itemize}

Suppose that for $n\in \N$ the space $Y$ admits a decomposition $Y=\bigcup_{i=1}^{M_n} Y_{n,i}$ as a union of non-empty compact  subsets $Y_{n,i}$, $i=1, \dots, M_n$, with  properties analogous to \textnormal{(i)}--\textnormal{(iii)} such that 
\begin{equation}\label{eq:iff1} X_{n+1,j} \sub  X_{n,i} \mbox{ if and only if }~Y_{n+1,j} \sub Y_{n,i}
\end{equation}
and 
 \begin{equation}\label{eq:iff2} X_{n,i}\cap X_{n,j} \neq \emptyset \mbox{ if and only if }~Y_{n,i} \cap Y_{n,j} \neq \emptyset
 \end{equation}  
for all $n$, $i$,  $j$.

Then there exists a unique homeomorphism $f\: X\ra Y$ such that 
$f(X_{n,i})=Y_{n,i}$ for all $n$ and $i$.  
\end{proposition}
In particular, under these assumptions the spaces $X$ and $Y$ are homeo\-morphic. 

\begin{proof}
We define a map $f\: X\ra Y$ as follows. For each point $x\in X$, by (ii) and (iii) there exists a nested sequence of sets $X_{n,i_n}$, $n\in \N$,  such that
$\{x \} = \bigcap_n X_{n,i_n}.$
Then the corresponding sets $Y_{n,i_n}$, $n\in \N$, are also nested by \eqref{eq:iff1}. Since these sets are non-empty and compact,  by condition (iii) for the space  $Y$ this implies  that there exists a unique point $y\in\bigcap_n Y_{n,i_n}$. We define $f(x)=y$.  

Then $f$ is well-defined. To see this,  suppose  we have  another nested sequence $X_{n,i'_n}$, $n\in \N$,  such that
$\{x \} = \bigcap_n X_{n,i_n'}.$ Then there exists a unique point $y'\in \bigcap_n Y_{n,i'_n}$. Now $x\in X_{n,i_n}\cap X_{n,i'_n}$ and so 
$ Y_{n,i_n}\cap Y_{n,i'_n}\ne \emptyset$  for all $n\in \N$ by \eqref{eq:iff2}. By condition (iii)  for $Y$, this is only possible if $y=y'$. So $f\: X\ra Y$ is indeed well-defined.

One can define a map $g\: Y\ra X$ by a similar procedure. Namely, for 
each $y\in Y$ we can find a nested sequence $Y_{n,i_n}$, $n\in \N$,  such that
$\{y \} = \bigcap_n Y_{n,i_n}.$
 Then there exists a unique point $x\in  \bigcap_n X_{n,i_n} $ and if we set $g(y)=x$, we obtain a well-defined map $g\: Y\ra X$. 
 
 It is obvious from the definitions that the maps $f$ and $g$ are inverse to each other.   Hence they define  bijections between $X$ and $Y$. 
 
 Conditions (i) and (ii) imply that if $X_{k,i}$ is a set in one of the decompositions of $X$ and $x\in X_{k,i}$, then there exists a nested sequence 
 $X_{n,i_n}$, $n\in \N$, with $X_{k,i_k}=X_{k,i}$ and $\{x \} = \bigcap_n X_{n,i_n}.$ This implies that $f(x)\in Y_{k,i}$ and so $f(X_{k,i})\sub Y_{k,i}$.
 Similarly, $g(Y_{k,i})\sub X_{k,i}$. Since $g=f^{-1}$, we have  $f(X_{k,i})=Y_{k,i}$ as desired. It is clear that this last condition together with our assumptions determines $f$ uniquely.

 It remains to show that $f$ is a homeomorphism. For this  it suffices to prove  that $f$ and $f^{-1}=g$ are continuous. Since the roles of $f$ and $g$ are completely symmetric, it is enough  to establish  that $f$ is continuous.
 
 For this,   let   $\e>0$ be arbitrary. By (iii) we can choose $n\in \N$ such that
$$\max \{ \diam (Y_{n, i}) : 1\leq i\leq M_n\} < \eps/2. $$
Since the sets $X_{n,i}$ are compact, there exists $\delta>0$ such that
$$\dist(X_{n,i}, X_{n,j}) >\delta,$$
whenever $i,j\in \{1, \dots, M_n\}$ and  $X_{n,i} \cap X_{n,j} = \emptyset$.

Now suppose that $a,b\in X$ are arbitrary points with  $d_X(a,b)< \delta $.
We claim that then  $d_Y(f(a),f(b))<\e$. Indeed, we can find $i,j\in \{1, \dots, M_n\}$ such that $a\in X_{n,i}$ and $b\in X_{n,j}$. Since 
$d_X(a,b)< \delta$, we then necessarily have $X_{n,i}\cap  X_{m,j}\ne \emptyset$ by definition of $\delta$.
So   $Y_{n,i}\cap  Y_{n,j}\ne\emptyset$ by \eqref{eq:iff2}. Moreover, $f(a)\in f(X_{n,i})=Y_{n,i}$ and $f(b)\in f(X_{n,j})=Y_{n,j}$.
Hence  
$$d_Y(f(a), f(b)) \leq \diam (Y_{n,i}) + \diam (Y_{n,j}) < \e.$$
The continuity of $f$ follows. \qed
\end{proof}

 \section{Topology of trees}\label{sec:toptree} 
 
  In this section we fix some terminology and collect some general facts about trees. We do not claim any originality of this material. All of it is standard and well-known, but we did not try to track it down    in the literature. Our objective is to make our presentation self-contained, and to have convenient reference points   for future work. For general background on trees or dendrites we refer 
 to \cite[Chapter V]{Wh}, \cite[Section \S 51 VI]{Ku68},  \cite[Chapter~X]{Na}), and the literature mentioned there.

 An {\em arc} $\alpha$ in a metric space  is a homeomorphic image of the  unit interval $[0,1]\sub \R$.  The points corresponding to $0$ and $1$ are called the {\em endpoints} 
 of $\alpha$. 
 
 Let  $T$ be a tree. Then the last part of Definition~\ref{def:tree} is equivalent to the requirement that for all points $a,b\in T$ with $a\ne b$, there 
 exists a unique arc in $T$  joining $a$ and $b$, i.e., it has the endpoints 
 $a$ and $b$. We use the notation $[a,b]$  for this unique arc. 
 It is convenient to allow $a=b$ here. Then $[a,b]$ denotes the {\em degenerate arc} consisting  only of the point $a=b$. 
 Sometimes we want to remove one or both endpoints from the arc $[a,b]$. Accordingly, we define
  $(a,b)=[a,b]\ba\{a,b\}$,  $[a,b)=[a,b]\ba\{b\}$ and 
  $(a,b]=[a,b]\ba\{a\}$.   In Section~\ref{s:CSST} we will not use this notation for arcs in a tree. There $[a,b]$ will always denote the Euclidean line segment joining  two points $a,b\in \C$.  
  
  A metric space $X$ is called {\em path-connected} if any two points $a,b\in X$ can be joined by a path in $X$, i.e., there exists a continuous map $\gamma\:[0,1]\ra X$ such that $\gamma(0)=a$ and $\gamma(1)=b$. The space $X$ is {\em arc-connected} if any two distinct points in $X$ can be joined by an arc in $X$. The image 
  of  a path joining two distinct points in a metric space always contains an arc joining these points (this  follows from  
the fact that every Peano space is arc-connected; see \cite[Theorem 3.15, p.~116]{HY}). In particular, 
every path-connected metric space is arc-connected. 
 
 \begin{lemma}\label{l:uni_cont} Let $(T,d)$ be a tree. Then for each $\eps>0$ there exists $\delta>0$ such that for all $a,b\in T$ with $d(a,b)<\delta$ we have $\diam([a,b])<\eps$.
\end{lemma}
\begin{proof} Fix $\eps>0$. Since $T$ is a compact, connected, and  locally connected  metric space, it is a {\em Peano space}. So by the Hahn-Mazurkiewicz theorem there exists a continuous surjective map $\varphi\:[0,1]\ra T$ of the unit interval onto $T$ \cite[Theorem 3.30, p.~129]{HY}. By uniform continuity of $\varphi$ we can represent $[0,1]$  as a union $[0,1]=I_1\cup \dots \cup I_n$ of finitely many 
closed intervals $I_1, \dots, I_n\sub [0,1]$  with $\diam(X_k)<\eps/2$,
where $X_k=  \varphi(I_k)$ 
for $k=1,\dots ,n$. The sets $X_k=\varphi(I_k) $ are compact. This implies that  there exists $\delta>0$ such that $\dist(X_i, X_j)>\delta$, whenever $i,j\in \{1,\dots, n\}$ and 
$X_i\cap X_j=\emptyset$. 

Now let $a,b\in T$ with $d(a,b)<\delta$ be arbitrary. We may assume $a\ne b$. 
Then there exist $i,j\in  \{1,\dots, n\}$ with $a\in X\coloneqq X_i$ and $b\in Y\coloneqq X_j$. 
By choice of $\delta$ we must have $X\cap Y\ne \emptyset$. As continuous images of intervals,
the sets $X$ and $Y$ are path-connected. Since $X\cap Y\ne \emptyset$, the union $X\cup Y$ that contains the points $a$ and $b$ is also path-connected. This implies that  $X\cup 
Y$ is arc-connected, and so there exists an arc $\alpha \sub X\cup Y$ with endpoints $a$ and $b$. 
The unique such arc in the tree $T$ is $[a,b]$, and so $[a,b]=\alpha\sub X\cup Y$. This implies 
$$ \diam ([a,b])\le \diam(X)+\diam(Y)<\eps,$$
as desired. \qed
 \end{proof}

 \begin{lemma} \label{lem:vartree} Let $(T,d)$ be a tree and $p\in T$. Then the following statements are true:
 \begin{itemize}
\item[\textnormal{(i)}]  Each  component $U$ of $T\ba\{p\}$ is an open and arc-connected subset of $T$.  

\smallskip 
\item[\textnormal{(ii)}] \ If  $U$ is a component of $T\ba\{p\}$, then $\overline U=U\cup\{p\}$ and  $\partial U=\{p\}$.

\smallskip  
\item[\textnormal{(iii)}] \ \ Two points $a,b\in T\ba\{p\}$ lie in the same component of $T\ba\{p\}$ if and only if $p\not\in [a,b]$. 
 \end{itemize}
  \end{lemma}
 
\begin{proof} (i) The set $T\ba\{p\}$ is open. Since $T$ is locally connected,  each component $U$ of $T\ba\{p\}$ 
is also open.   

For $a,b\in U$ we write $a\sim b$ if $a$ and $b$ can be joined by a path in $U$. Obviously,
this defines an equivalence relation on $U$. The equivalence classes are open subsets
of $T$. To see this,  suppose  $a,b\in U$ can be joined by a path $\beta$ in $U$. Then for all points 
$x$ in a sufficiently small neighborhood $V\sub U $ of $b$ we have $[b,x]\sub U$ as follows from  Lemma~\ref{l:uni_cont}. 
So by  concatenating $\beta$ with (a parametrization of) the arc $[b,x]$, we obtain a
path $\beta'$ in $U$ that joins $a$ and $x\in V$. This shows that every point $b$ in the equivalence class of $a$ has a neighborhood $V$ that also belongs to this equivalence class. 

We see that the equivalence classes of $\sim$ partition $U$ into open sets. Since $U$ is connected, there can only be one such set. It follows that  $U$ is path-connected and 
hence also arc-connected.

\smallskip
(ii) Let $U$ be a (non-empty) component of $T\ba\{p\}$. We choose  a point $a\in U$. 
The set  $[a,p)$ is  connected, contained in $T\ba\{p\}$, and meets $U$ in $a$. Hence $[a,p)\sub U$. This implies that $p\in \overline U$. On the other hand, the set 
$U\cup \{p\}$ is closed, because its complement is a union of components of $T\ba\{p\}$ and hence open by (i). Thus  $\overline U=U\cup\{p\}$. By (i) no point in $U$ is a boundary point of $U$, and so $\partial U=\{p\}$.

\smallskip
(iii) If $a,b\in T\ba\{p\}$ and $p\not \in [a,b]$, then $[a,b]$ is a  connected subset of  
$T\ba\{p\}$. Hence $[a,b]$ lies in a component $U$ of $T\ba\{p\}$. In particular, $a,b\in [a,b]$ lie in the same component $U$ of $T\ba\{p\}$.

Conversely, suppose that $a,b\in T\ba\{p\}$ lie in the same component $U$ of 
$T\ba\{p\}$. We know by (i) that $U$ is arc-connected. Hence there exists a (possibly degenerate) arc $\alpha\sub U$ with endpoints $a$ and $b$. But the unique such arc in $T$ is $[a,b]$. Hence $[a,b]=\alpha\sub U\sub T\ba\{p\}$, and so $p\not\in [a,b]$. \qed
\end{proof} 

A subset $S$ of a tree $(T,d)$ is called a {\em subtree} of $T$ if $S$ equipped with the restriction of the metric $d$ is also a tree as in Definition~\ref{def:tree}. Every subtree $S$ of $T$ contains 
two points and hence a non-degenerate arc. In particular, every subtree $S$ of $T$ is an infinite, actually uncountable set.

The following statement characterizes 
 subtrees.

 \begin{lemma} \label{lem:subtree} Let $(T,d)$ be a tree. Then a set  
 $S\sub T$ is a subtree of $T$ if and only if $S$ contains at least two points and is  closed and connected.  
 \end{lemma} 
 
 \begin{proof} 
 If $S$ is a subtree of $T$, then $S$ contains at least two points, and is connected and 
 compact. Hence  it is a closed subset of $T$. Conversely, suppose that $S$ contains at least two points and is closed and connected. Then $S$ is compact, because $T$ is compact.
 
 Suppose that $a,b\in S$, $a\ne b$, are two distinct points in $S$. 
 We consider the arc $[a,b]\sub T$. Suppose there exists a point 
 $p\in [a,b]$ with  $p\not\in S$. Then $p\ne a,b$, and so by Lemma~\ref{lem:vartree}~(iii), the points $a$ and $b$ lie in different components of $T\ba \{p\}$. This is impossible, because the connected set $S\sub T\ba \{p\}$ must be  contained in exactly one component of $T\ba \{p\}$. This shows that $[a,b]\sub S$ and so 
 the points $a$ and $b$  can be joined by an arc in $S$. This arc in $S$ is unique, because it is unique in $T$. 
 
 It remains to show that $S$ is  locally connected, i.e., every point in $S$ has \arbly\ small connected relative neighborhoods. To see this,   let $a\in S$ and $\eps>0$ be \arb. Then by Lemma~\ref{l:uni_cont}
 we can find $\delta>0$ such that $[a,x]\sub B(a,\eps)$ whenever 
 $x\in B(a,\delta)$. Now let $M$ be the union of all arcs $[a,x]$ with $x\in S\cap B(a,\delta)$. These arcs lie in $S$ and so $M$ is a connected set contained in $S\cap B(a,\eps)$. Moreover, 
 $S\cap B(a,\delta)\sub M$ and so $M$ is a connected relative 
(not necessarily open) neighborhood of $a$ in $S$. This shows that $S$ is locally connected. We conclude that $S$ is indeed a subtree of $T$.     \qed
\end{proof}

\begin{lemma} \label{lem:branch} Let $(T,d)$ be a tree, $p\in 
T$, and $U$ a component of $T\ba\{p\}$. Then $B=U\cup\{p\}$ is a subtree of $T$ and $p$ is a leaf of $B$.  \end{lemma} 

\begin{proof} It follows from Lemma~\ref{lem:vartree} (i) and (ii) 
that  the set  $U$ is connected and  that $B= U\cup\{p\}=\overline U$. This implies that $B$ is  closed and connected. Since $U\ne \emptyset$ and $p\not\in U$, the set $B$ contains at least two points. Hence $B$ is  a subtree of $T$ by Lemma~\ref{lem:subtree}. Since $B\ba\{p\}=U$ is connected, $p$ is a leaf of $B$.   \qed
\end{proof} 

If the subtree $B=U\cup\{p\}$ is as in the previous lemma, then we call $B$  a {\em branch} of $p$ in $T$ (or just a branch of $p$ if $T$ is understood).  

\begin{lemma}\label{lem:subbranches}
Let $(T,d)$ be a tree, $S\sub T$ be a subtree of $T$, and $p\in S$.
Then every branch $B'$ of $p$ in $S$ is contained in a unique branch $B$ of $p$ in $T$. The  
 assignment $B'\mapsto B$ is an injective map between the  sets of branches of $p$  in $S$ and in $T$. If $p$ is an interior point of $S$, then this map is a bijection. \end{lemma}

In particular, if under the given assumptions $\nu_T(p)$ is the valence of $p$ in $T$ and $\nu_S(p)$ the valence of $p$ in $S$, then $\nu_S(p)\le \nu_T(p)$. Here we have equality if $p$ is an interior point of $S$. 

If $p\in S$ is a leaf of $T$, then $T$ has only one branch $B$ at $p$, namely $B=T$. Hence  $1\le \nu_S(p)\le \nu_T(p)\le 1$, and so  $\nu_S(p)=1$. This means that $p$ is also a leaf of $S$. More informally, we can say that the property of a point 
being a leaf in $T$ is passed to subtrees that contain the point.

\begin{proof} If $B'$ is a branch of $p$ in $S$, then $B'=U'\cup \{p\}$, where $U'$ is a component of $S\ba\{p\}$. Then $U'$ is a connected subset of $T\ba\{p\}$ and so contained in a unique component $U$ of $T\ba\{p\}$. Then $B=U\ \cup \{p\}$ is a branch of $p$ in $T$ with $B'\sub B$ and it is clear that $B$ is the unique such branch.

To show injectivity of the map $B'\mapsto B$, let   $B_1'$ and $B_2'$ be two distinct branches of $p$ in $S$. 
Pick points $a\in B_1'\ba\{p\}$ and $b\in B_2'\ba \{p\}$. Then $a$ and $b$ lie in different components of $S\ba\{p\}$ and so $p\in [a,b]$ by Lemma~\ref{lem:vartree}~(iii) applied to the tree  $S$. 
Hence $a$ and $b$ lie in different components of $T\ba\{p\}$, and so in different branches of $p$ in $T$. This implies that 
$B_1' $ and $B_2'$ must be contained  in different branches of $p$ in $T$. This shows that the map $B'\mapsto B$ is indeed injective. 

Now assume in addition that $p$ is an interior point of $S$. 
To show surjectivity of the map $B'\mapsto B$, we consider a branch $B$ of $p$ in $T$. Pick a point $a\in B\ba\{p\}$. Then $[a,p)\sub B\ba\{p\}$, because $B$ is a subtree of $T$. Since $p$ is an interior point 
of $S$, there exists a point $x\in [a,p)$ close enough to $p$ such that $x\in S\ba\{p\}$. If $B'$ is the unique branch of $p$ in $S$ that contains $x$, then we have $x\in B'\cap B$. This  implies $B'\sub B$. Hence the map $B'\mapsto B$ is also surjective, and so a bijection.  \qed
\end{proof} 

\begin{lemma}\label{lem:tripcrit}
Let $(T,d)$ be a tree, $p,a_1,a_2,a_3\in T$ with $p\ne a_1, a_2, a_3$ and suppose that the sets $[a_1,p)$, $[a_2,p)$, $[a_3,p)$
are pairwise disjoint. 
Then  the points $a_1,a_2,a_3$ lie in different components of $T\ba\{p\}$ and $p$ is a branch point of $T$.  
\end{lemma}

\begin{proof} The arcs $[a_1,p]$ and $[a_2,p]=[p, a_2]$ have only the point $p$ in common. So   their union  $[a_1,p]\cup [p, a_2]$ 
is an arc and this arc must be equal to $[a_1,a_2]$. Hence $p\in [a_1,a_2]$ which by Lemma~\ref{lem:vartree}~(iii) implies that $a_1$ and $a_2$ lie in different components of $T\ba\{p\}$. A similar argument shows that $a_3$ must be contained in a component of 
$T\ba\{p\}$ different from the components containing $a_1$ and  $a_2$. In particular, $T\ba\{p\}$ has at least three components and so $p$ is a branch point of $T$. The statement follows.  \qed
\end{proof} 

\begin{lemma}\label{triple_dense}
Let $(T,d)$ be a tree such that the branch points of $T$ are dense in $T$. If $a,b \in T$ with $a\ne b$, then there exists a branch  point $c\in (a,b)$.
\end{lemma}
\begin{proof}
We pick a point $x_0 \in (a,b)\ne \emptyset$. 
Then $x_0$ has positive distance to both $a$ and $b$. This and Lemma~\ref{l:uni_cont} imply that we can find $\delta>0$ such that 
for all $x\in B(x_0, \delta)$ the arc $[x,x_0]$ has uniformly small diameter and so does not contain $a$ or $b$.   
  
Since branch  points are dense in $T$, we can find a branch  point $p\in B(x_0,\delta)$. Then $a,b\not\in [p,x_0]$. If $p\in (a,b)$, we are done.

In the other case, we have  $p\not \in (a,b)$. 
 If we travel from $p$ to $x_0\in (a,b)$ along $[p,x_0]$, we meet 
 $[a,b]$ in a first point 
$c\in (a,b)$. Then $a,b,p\ne c$. Moreover, the sets $[a,c)$, $[b,c)$, $[p,c)$ are pairwise disjoint. Hence $c\in (a,b)$ is a branch  point of $T$ as follows from 
 Lemma~\ref{lem:tripcrit}.   \qed
\end{proof}

\begin{lemma}\label{lem:brfinite} Let $(X,d)$ be a compact, connected,  and locally connected metric space, $J$ an index set,
$p_i\in T$, and $U_i$ a  component of $X\ba\{p_i\}$ for 
each $i\in J$. Suppose that 
$$U_i\cap U_j=\emptyset $$ for all $i,j\in J$, $i\ne j$.  Then  $J$ is a countable set. If there exists $\delta>0$ such that 
$\diam(U_i)>\delta$ for each $i\in J$, then $J$ is finite.   
\end{lemma} 

Informally, the space $X$ cannot contain a ``comb" with too many long teeth.

\begin{proof} We prove the last statement first. We argue by contradiction and assume that    $\diam(U_i)>\delta>0$ for each $i\in J$, where $J$ is an infinite index set.  Then we can choose a point $x_i\in U_i$ such that $d(x_i,p_i)\ge \delta/2$. 
The set $A=\{x_i:i\in J\}$ is infinite and so it must have a limit point $q\in X$, because $X$ is compact.   Since $X$ is locally connected,
there exists a connected neighborhood $N$ of $q$ such that 
$N\sub B(q,\delta/8)$. Since $q$ is a limit point of $A$, the set $N$ contains infinitely many points in $A$. In particular, we can find 
$i,j\in J$ with $x_i,x_j\in N$ and $i\ne j$. 
Then 
$$\dist(p_i, N)\ge  d(p_i, x_i)-\diam(N)\ge \delta/2-\delta/4>0,$$
and so $N\sub X\ba\{p_i\}$. Since the  connected set $N$ meets  $U_i$ in the point $x_i$, this implies that $N\sub U_i$. Similarly, $N\sub U_j$. This is impossible, because we have $i\ne j$ and so $U_i\cap U_j\ne \emptyset$, while $\emptyset\ne N\sub U_i\cap U_j$.

To prove the first statement, note that $\diam(U_i)>0$ for each $i\in J$. Indeed, otherwise  $\diam(U_i)=0$ for some $i\in J$. Then  $U_i$ consists of only one point $a$. Since $X$ is locally connected, the component $U_i$ of $X\ba\{p_i\}$ is an open set. So $a$ is an isolated point of $X$. This is impossible, because the metric space $X$ is connected and so it does not have isolated points. 

Now we write  $J=\bigcup_{n\in \N}J_n$, where 
$J_n$ consists of all $i\in J$ such that $\diam(U_i)>1/n$. Then each set $J_n$ is finite by the first part of the proof. This implies that $J$ is countable.  \qed
 \end{proof}

We  can apply the previous lemma to a tree $T$ 
and choose for each $p_i$ a fixed branch point $p$ of $T$. Then it follows that $p$ can have at most countably many distinct 
complementary components $U_i$ and hence there are only countably many 
distinct branches $B_i=U_i\cup\{p\}$ of  $p$. Moreover, since $\diam(B_i)=\diam(\overline U_i)=\diam(U_i)$, there can only be  finitely many of these branches whose diameter  exceeds a given positive number $\delta>0$. In particular, we can label the branches of $p$ by numbers $n=1,2,3,\dots$ so that 
$$ \diam(B_1)\ge \diam(B_2)\ge \diam(B_3)\ge \dots\, . $$
 We now set 
 \begin{equation}\label{eq:height}
 H_T(p)=\diam(B_3)
 \end{equation}  and call $H_T(p)$  the {\em height} of the branch point $p$ in $T$. So the height of a branch point $p$ is the diameter of the third largest branch of $p$.  


\begin{lemma} \label{finite} Let $(T,d)$ be a tree and 
$\delta>0$. Then  there are at most  finitely many  branch  points $p\in T$ with height $H_T(p)>\delta$.
 \end{lemma}
\begin{proof} We argue by contradiction and assume that this is not true.
Then the  set $E$ of branch  points $p$ in $T$ with $H_T(p)>\delta$ has  infinitely many  
elements. Since $T$ is compact, the set $E$  has a limit point   $q\in T$. 

\smallskip
{\em Claim.} There exists a branch $Q$ of $q$ 
 such that  the set $E\cap Q$ is infinite and has  $q$ as a limit point.
  
  \smallskip
 Otherwise, $q$ has infinitely many distinct branches  $Q_n$, $n\in \N$, that  contain a point $a_n\in  E  \cap (Q_n\ba\{q\})$. Then $a_n$ is a branch point with $H_T(a_n)> \delta$ which implies 
 that $a_n$ has at least three branches whose diameters exceed $\delta$.  At least one of them does not contain $q$. If we denote such a  branch of $a_n$ by  $V_n$, then $V_n$  is a connected subset of  $T\ba\{q\}$.  It meets $Q_n\ba\{q\}$, because $a_n\in (Q_n\ba\{q\})\cap V_n$. It follows that $V_n\sub Q_n$ and so $\diam(Q_n)\ge \diam (V_n)>\delta$. Since the branches $Q_n$  of $q$ are all 
 distinct for $n\in \N$, this contradicts Lemma~\ref{lem:brfinite} (see the discussion after the proof of this lemma).
 The Claim follows.
 
 \smallskip    
 We fix a branch $Q$ of $q$ as in the Claim. 
For each $n\in \N$ we will now inductively construct branch points  $p_n\in E\cap (Q\ba\{q\})$  together with a branch $B_n$ of  $p_n$ and an auxiliary compact set $K_n\sub T$. 
They will satisfy the following conditions for each $n\in \N$: 
\begin{itemize}
\item[\textnormal{(i)}] $\diam(B_n)>\delta$,

\smallskip 
\item[\textnormal{(ii)}] \ the sets $B_1, \dots, B_n$ are disjoint, 

\smallskip 
\item[\textnormal{(iii)}] \ \ the set $K_n$ is compact and connected, 
and $$B_1\cup \dots \cup B_n\sub K_n \sub Q\ba \{q\}.$$  
\end{itemize}

We pick  an arbitrary branch  point $p_1\in E\cap (Q\backslash \{q\})$ to start. Then we can choose a branch  $B_1$ of $p_1$ that does not contain $q$ and satisfies $\diam(B_1)>\delta$. We set $K_1=B_1$. Then $K_1$ is a compact and connected set 
that does not contain $q$ and meets $Q$, because $p_1\in K_1\cap Q$. Hence $K_1 \sub Q\ba \{q\}$. 

Suppose for some  $n\in \N$, a branch point  $p_k\in E\cap Q$, a branch $B_k$ of $p_k$, and a set 
$K_k$ with the properties (i)--(iii) have been chosen for all $1\le k\le n$. 

Since $q\not\in K_n$, we have $\dist(q,K_n)>0$, and so we can find  a branch  point $p_{n+1}\in E\cap (Q\ba\{p\})$ sufficiently close to $q$ such that 
$p_{n+1}\not\in K_n$. This is possible, because $q$ is a limit point of $E\cap (Q\ba\{q\})$.  Since the set $K_n\sub T\ba\{p_{n+1}\}$ is connected, it must be contained in a branch of $p_{n+1}$. Since there are three branches 
of $p_{n+1}\ne q$ whose diameters exceed $\delta$, we can pick one of them that contains neither $q$ nor $K_n$. Let $B_{n+1}$ be such a branch of $p_{n+1}$. Then $\diam(B_{n+1})>\delta$ and so (i) is true for $n+1$.  We have  
$B_{n+1}\cap K_n=\emptyset$; so (iii) shows that  $B_{n+1}$ is disjoint from the previously chosen disjoint sets $B_1,\dots, B_n$. This gives (ii).

Since $p_n, p_{n+1}\in Q\ba \{q\}$, the arc $[p_n,p_{n+1}]$ does not contain $q$ (see Lemma~\ref{lem:vartree}~(iii)). We also have  $p_n\in B_n\sub K_n$ and $p_{n+1}\in B_{n+1}$, which implies that the  
 set $K_{n+1}\coloneqq K_{n} \cup [p_n, p_{n+1}]\cup B_{n+1}\sub  Q\ba \{q\}$ is compact and connected. We  have  $$B_1\cup \dots \cup B_n\cup B_{n+1}\sub K_n \cup B_{n+1} \sub K_{n+1} \sub
 Q\ba \{q\},$$
and so $K_{n+1}$ has  property (iii). 

 Continuing with this process, 
 we  obtain  disjoint branches $B_n$ for all $n\in \N$ that  satisfy (i). 
The last part of Lemma~\ref{lem:brfinite} implies  that this is impossible and we get a contradiction.  \qed
\end{proof}

\section{Basic properties of the  continuum self-similar tree} \label{s:CSST}

We now  we study the  properties of  the continuum self-similar tree (CSST). Unless otherwise specified, all metric notions in this section refer to the Euclidean metric on the complex plane $\C$. In this section, $i$ always denotes the imaginary unit and we do not use this letter for indexing as in the other sections.   If $a,b\in \C$ we denote by $[a,b]$ the Euclidean line segment in $\C$ joining $a$ and $b$. We also use the usual notation for open or half-open line segments. 
So $[a,b)=[a,b]\ba\{b\}$, etc. 

%
%
%


For the proof of  Proposition~\ref{CSSTex} we consider a coding 
procedure of certain points in the complex plane by words in an alphabet. We first fix some terminology related to this. 
We consider a non-empty set $\mathcal{A}$. Then we call $\mathcal{A}$ an {\em alphabet} and refer to the elements in $\mathcal{A}$ as the {\em letters}  in this alphabet. In this paper 
we will only use alphabets of the form  
$\mathcal{A}=\{1,2, \dots ,m\}$ with $m\in \N$, $m\ge 3$.  
We consider the set 
$W(\mathcal{A})\coloneqq \mathcal{A}^{\mb{N}}$ of infinite sequences in $\mathcal{A}$  as the set of {\em infinite words} in the alphabet $\mathcal{A}$ and write the elements $w\in	 W(\mathcal{A})$ in the form $w=w_1w_2\ldots$, where it is understood that $w_k\in 
\mathcal{A}$ for $k\in \N$. Similarly,  we set 
$W_n(\mathcal{A})\coloneqq \mathcal{A}^n$ and consider 
$W_n(\mathcal{A})$ as the set of all words in the alphabet 
$\mathcal{A}$ of length $n$. We write the elements $w\in W_n(\mathcal{A})$ in the form $w=w_1\dots w_n$ with $w_k\in  \mathcal{A}$ for $k=1, \dots , n$. 
We use the convention that $W_0( \mathcal{A})=\{\emptyset\}$
and consider the only element $\emptyset$ in $W_0( \mathcal{A})$ as the {\em empty word} of length $0$.  Finally, 
$$W_*(\mathcal{A})\coloneqq \bigcup_{n\in \N_0}  W_n(\mathcal{A})$$ is the set of all words of finite length. 
If $u=u_1\dots u_n$ is a finite word and $v=v_1v_2\dots $ is a finite or infinite word in the alphabet $\mathcal{A}$, then  we denote by $uv=u_1\dots u_nv_1v_2\dots$ the word obtained by concatenating $u$ and $v$. We call $u$ an {\em initial segment} and $v$ a {\em tail} of the word $w=uv$.  If the alphabet $\mathcal{A}$ is understood, then we will simply drop $\mathcal{A}$ from the notation. So $W$ will denote the set of infinite words in $\mathcal{A}$, etc. 

For the rest of this section, we use the alphabet $\mathcal{A}=\{1,2,3\}$. So when we  write  $W$, $W_n$, $W_*$ it is understood that $\mathcal{A}=\{1,2,3\}$ is the underlying alphabet. 
There exists a unique metric $d$ on $W=\{1,2,3\}^{\N}$ with the following property. If we have two words $u=u_1u_2\ldots$ and  
$v=v_1v_2\ldots$ in $W$ and $u\ne v$, then for some  $n\in \N_0$ we have  $u_1=v_1$, \dots, $u_{n}=v_{n}$,  and $u_{n+1}\ne v_{n+1}$.  Then $d(u,v)=1/2^n$. More informally, two elements $u,v\in W$ are close in this metric precisely if  they share a large 
number of initial letters. The metric space $(W, d)$ is compact and homeomorphic to a Cantor set. 

If $n\in \N_0$ and $w=w_1w_2\dots w_n\in W_n$, we define 
$$f_w \coloneqq  f_{w_1}\circ f_{w_2}\circ \cdots \circ f_{w_n},$$
 where we use the maps in \eqref{maps} in the composition.  By convention, 
 $f_\emptyset= \text{id}_\C$ is the identity map on $\C$.
 Note that $f_w$ is a Euclidean similarity on $\C$ that scales Euclidean distances by the factor $2^{-n}$. If $a,b\in \C$, then 
 $f_w([a,b])=[f_w(a),f_w(b)]$. We will use this repeatedly in the following.

Throughout this section we denote by $H\sub \C$  the (closed) convex hull of the four points 
$1$, $i$, $-1$, and $\tfrac{1}{2}-\frac{i}{2}$ (see Figure~\ref{f:Hsets}). We set $H_k=f_k(H)$ for $k=1,2,3$. Then 
$$H_1\cup H_2 \cup H_3 =f_1(H)\cup f_2(H) \cup f_3(H)\sub H.$$
This implies that 
\begin{equation}\label{Hinv}
f_w(H)\sub H
\end{equation}  for all 
$w\in  W_{\ast}$.

 \begin{figure}
 \begin{overpic}[ scale=0.7
    ]{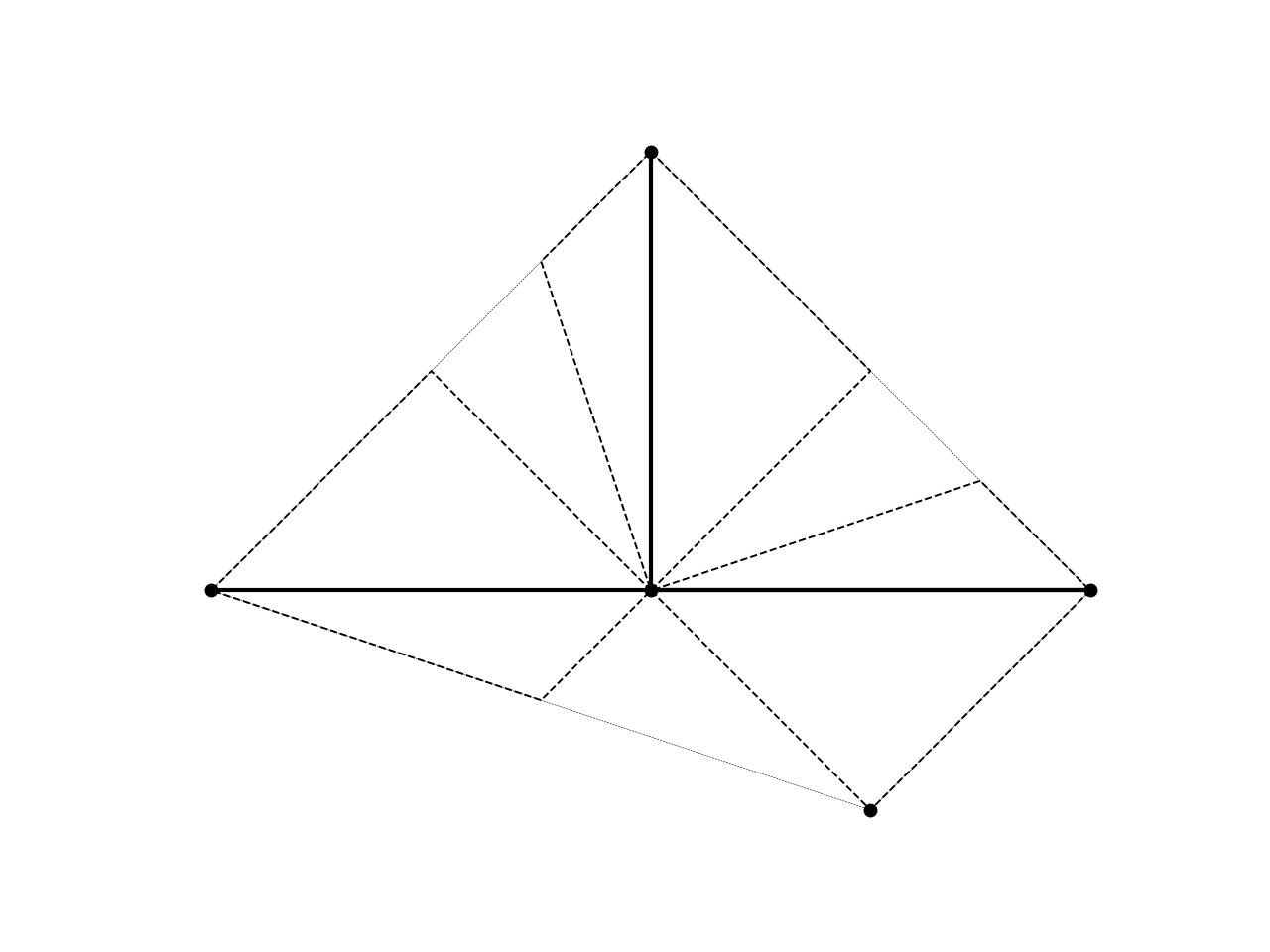}
    \put(50,65){ $i$}
    \put(14,25){ $-1$}
     \put(85,25){ $1$}
       \put(85,25){ $1$}
          \put(70, 10){ $-\frac 12 +\frac i2$}
            \put(50, 25){ $0$}
             \put(30, 35){ $H_1$}
                \put(70, 20){ $H_2$}
                  \put(55, 50){ $H_3$}
              \put(40, 25){ $I_1$}
               \put(60, 25){ $I_2$}
                \put(53, 37){ $I_3$}
                  \put(75, 30){ $I$}
               \put(35, 15){ $H$}
  \end{overpic}
  \vspace{-1cm}
\caption{Ilustration of some associated sets.}
\label{f:Hsets}
\end{figure}
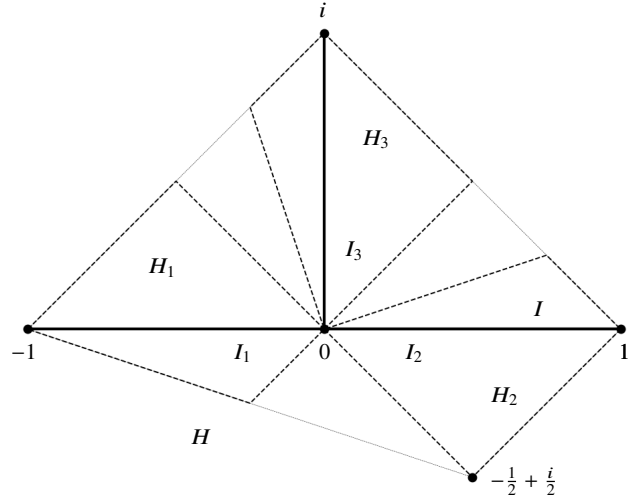

\begin{lemma} \label{piex} There exists a well-defined continuous map 
$\pi\: W\ra \C$  given  by
  $$\pi(w)=\lim_{n\to \infty} f_{w_1w_2\dots w_n}(z_0)$$
  for $ w=w_1w_2\ldots\in W $ and $z_0\in \C$.
Here the limit exists and is independent of the choice of $z_0\in \mb{C}$. 
 \end{lemma} 
 
The existence of such a map $\pi$ is standard in similar  contexts 
(see, for example,  \cite[Section 3.1, pp.~426--427]{Hu81}).
 In the following, $\pi\: W\ra \C$ will always denote the map provided by this lemma.
 
 \begin{proof} Fix $z_0\in \C$. Then 
 there exists a constant $C\ge 0$ such that 
 $$ |z_0-f_k(z_0)| \le C$$ for $k=1,2,3$.
 If  $n\in \N_0$ and $u\in W_n$, then 
 $$ |f_u(a)-f_u(b)|=\frac 1{2^{n}} |a-b|$$ for all $a,b\in \C$. 
 This implies that if $w=w_1w_2\ldots\in W$, $n\in \N$,  and $u\coloneqq w_1w_2\dots w_n\in W_n$, then 
 \begin{align*}
 |f_{w_1w_2\dots w_{n}}(z_0)-f_{w_1w_2\dots w_{n+1}}(z_0)|&
 =|f_u(z_0) -f_u(f_{w_{n+1}}(z_0))|\\
 &=\frac 1{2^{n}} | z_0-f_{w_{n+1}}(z_0)|\le \frac C{2^{n}}.
 \end{align*} 
 It follows  that  $\{f_{w_1w_2\dots w_{n}}(z_0)\}_{n\in \N}$ is a Cauchy sequence in $\C$. Hence this sequence converges and 
 $$\pi(w)=\lim_{n\to \infty} f_{w_1w_2\dots w_n}(z_0)$$
 is well-defined for each $w=w_1w_2\ldots\in W$. 
 
 The limit does not depend on the choice of $z_0$. Indeed, if $z'_0\in \C$ is another point, then 
 $$ |f_{w_1w_2\dots w_n}(z_0)-f_{w_1w_2\dots w_n}(z'_0)|
 =\frac1{2^{n}} |z_0-z'_0|,$$ which implies that 
  $$\lim_{n\to \infty} f_{w_1w_2\dots w_n}(z_0)=
  \lim_{n\to \infty} f_{w_1w_2\dots w_n}(z'_0). $$
  
  The definition of $\pi$ shows  that if $w=w_1w_2\ldots\in W$ and $n\in \N_0$, then 
    \begin{equation} \label{pifrel}\pi(w)=\pi(w_1w_2\ldots)=f_{w_1\dots w_n}(\pi(w_{n+1}w_{n+2}\ldots)). 
    \end{equation}
    If we pick $z_0\in H$, then \eqref{Hinv} and the definition of $\pi$ imply  that $\pi(W)\sub H$.  If we combine this with \eqref{pifrel}, then we see that if two words $u,v\in W$ start with the same letters 
    $w_1,\dots,  w_n$, then 
    $$ |\pi(u)-\pi(v)|\le  \diam(f_{w_1\dots w_n}(H))=\frac1{2^{n}} \diam(H).$$ 
    The continuity of the map $\pi$ follows from this and the definition of the metric $d$ on $W$. \qed
 \end{proof} 
 We can now establish  the result that is the basis of the definition of 
 the CSST. Again arguments along these lines are completely standard. 
  
\begin{proofof}{\em Proof of Proposition~\ref{CSSTex}. } Let  $\pi\: W\ra \C$ be  the map provided by  Lem\-ma~\ref{piex} and  define $\mb{T}=\pi(W)\sub \C$. Since $W$ is compact and $\pi$ is 
 continuous, the set $\mb{T}$ is non-empty and compact. The relation \eqref{Tinv} immediately follows from \eqref{pifrel} for $n=1$. Note  that \eqref{Tinv} implies that
 \begin{equation}\label{Tinv2}
 f_w( {\mb{T}})=f_{w1}( {\mb{T}})\cup f_{w2}( {\mb{T}})\cup f_{w3}( {\mb{T}})
\end{equation}
for each $w\in W_n$, $n\in \N_0$.  From this in turn we deduce that 
\begin{equation}\label{Tinv3}
 \bigcup_{w\in W_n} f_{w}( {\mb{T}})=  {\mb{T}} 
\end{equation}
for each $n\in \N_0$.

 It remains to show the uniqueness of $\mb{T}$. Suppose $\widetilde {\mb{T}}\sub \C$ is another non-empty compact set satisfying the analog of   \eqref{Tinv}. Then the analogs of 
 \eqref{Tinv2} and \eqref{Tinv3} are also valid for $\widetilde {\mb{T}}$.  
 This and  the definition of 
 $\pi$ using  a point $z_0\in \widetilde {\mb{T}}$ imply that $\mb{T} =\pi(W)\sub
  \widetilde {\mb{T}}$. 
  
  For the converse inclusion, let $a\in  \widetilde {\mb{T}}$ be arbitrary. Using the relation 
  \eqref{Tinv2}  for the set   $\widetilde {\mb{T}}$, we can inductively  construct an infinite word $w_1w_2\ldots \in W$ such that 
   $ a\in f_{w_1w_2\dots w_n}(\widetilde {\mb{T}}) $
   for all $n\in \N$. Since 
  $$ \diam( f_{w_1w_2\dots w_n}(\widetilde {\mb{T}}))=\frac1{2^n}
   \diam( \widetilde {\mb{T}})\to 0 \text{ as $n\to \infty$},$$
   the definition of $\pi$ (using a point $z_0\in \widetilde {\mb{T}}$) implies that 
   $a=\pi(w)$. In particular, $a\in \pi(W)=\mb{T}$, and so $\widetilde {\mb{T}}\sub \mb{T}$. The uniqueness of $\mb{T}$ follows.  \qed
 \end{proofof}
 In the proof of the previous proposition we have seen that 
$ \mb{T} =\pi(W)$. If $p\in \Te$ and $p=\pi(w)$ for some $w\in W$, then we say that the word $w$ {\em represents} $p$. 

The following statement provides some  geometric descriptions of $ \mb{T} $.

\begin{proposition}\label{prop:CSSTgeo} Let $I=[-1,1] \sub \C$. 
For $n\in \N_0$ define 
\begin{equation*}
 J_n= \bigcup_{w\in W_n} f_w(I) \quad \text{and} \quad K_n= \bigcup_{w\in W_n} f_w(H) .
\end{equation*}
 Then the sets $J_n$ and $K_n$ are compact and   satisfy  
\begin{equation}\label{treeincl} J_{n}\subseteq J_{n+1}\sub \mb{T} \sub K_{n+1}\sub K_n
\end{equation}
for $n\in \N_0$.
Moreover, we have  
\begin{equation}\label{treerep} 
 \overline {\bigcup_{n\in \N_0}J_n}=  \mb{T}=\bigcap_{n\in \N_0}K_n. 
\end{equation}
\end{proposition}

As we will discuss more towards the end of this section, the first identity 
in \eqref{treerep} represents $\Te$ as the closure of a union of 
an ascending  sequence of trees as mentioned in the introduction.
We will not need   the second  identity  
in \eqref{treerep} in the following, but included it to show that 
$\Te$ can also be obtained as the intersection of a natural decreasing sequence of compacts sets. This is how many other fractals are constructed.  

\begin{proof} It is clear that the sets $J_n$ and $K_n$ as defined in the statement are compact for each $n\in \N_0$. Set $I_k=f_k(I)$ for $k=1,2,3$. Then an elementary geometric consideration shows that 
(see Figure~\ref{f:Hsets})
$$I\sub I_1\cup I_2\cup I_3 \sub H_1\cup H_2 \cup H_3
 \sub {H}.$$ This in turn implies that 
\begin{align*}
 f_w(I) &\sub f_{w1}(I)\cup  f_{w2}(I)\cup  f_{w3}(I)\\
 &\sub 
 f_{w1}(H)\cup  f_{w2}(H)\cup  f_{w3}(H)\sub  f_w(H) 
\end{align*} 
for each $w\in W_n$, $n\in \N_0$. Taking the union over all $w\in W_n$, we obtain 
\begin{equation}\label{impincl}
 J_n\sub J_{n+1} \sub K_{n+1} \sub K_{n}
\end{equation}
for all $n\in \N_0$. 
The set $\widetilde {\mb{T}}= \overline {\displaystyle \bigcup_{n\in \N_0}J_n}$ is non-empty, compact, and satisfies
\begin{align*}
\bigcup_{k=1,2,3} f_k(\widetilde {\mb{T}})&= 
\bigcup_{k=1,2,3} f_k\biggl( \overline { \bigcup_{n\in \N_0}J_n }
\biggr)=  \bigcup_{k=1,2,3}  \overline {  f_k\biggl( \bigcup_{n\in \N_0}J_n\biggr) } \\
&  =\overline {  \bigcup_{k=1,2,3}   f_k\biggl( \bigcup_{n\in \N_0}J_n\biggr)}= \overline {   \bigcup_{n\in \N_0} \bigcup_{k=1,2,3}   f_k(J_n)}\\
& = \overline {   \bigcup_{n\in \N_0} J_{n+1} } = \overline {   \bigcup_{n\in \N_0} J_{n} }  =\widetilde {\mb{T}}.
\end{align*} 
Hence $\widetilde {\mb{T}}=\mb{T}$ by the uniqueness statement in Proposition~\ref{CSSTex}.  So we have the first equation in \eqref{treerep}.

Since $0\in H$, we have $f_w(0)\in f_w(H)\sub K_n$ for each 
$w\in W_n$. Since the sets $K_n$ are compact and nested, this implies that for each $w=w_1w_2\ldots \in W$ we have 
$$ \pi(w)=\lim_{n\to \infty} f_{w_1\dots w_n}(0)\in \bigcap_{n\in \N_0}
K_n. $$ 
It follows that $\mb{T} =\pi(W)\sub \displaystyle \bigcap_{n\in \N_0} K_n$. 

To show the reverse inclusion, let $a\in  \displaystyle \bigcap_{n\in \N_0} K_n$
be arbitrary. Then $a\in K_n$ for each $n\in \N_0$, and so there is a word $u_n\in W_n$ such that $a\in f_{u_n}(H)$.
Define $z_n= f_{u_n}(0)\in J_n \sub \mb{T}$.
Since $0\in H$, we  have $z_n\in  f_{u_n}(H)$, and so 
$$ |z_n-a|\le \diam(f_{u_n}(H))=\frac{1}{2^n} \diam(H). $$
Hence $z_n\to a$ as $n\to \infty$. Since $z_n\in \mb{T} $ and $\mb{T}$ is compact, it follows that $a\in \mb{T}$.
We see that  $\displaystyle \bigcap_{n\in \N_0} K_n\sub \mb{T}$.
So the second equation in \eqref{treerep} is also valid.

The inclusions  \eqref{treeincl} follow from \eqref{treerep} and
\eqref{impincl}.  \qed
 \end{proof} 

For a finite word $u\in W_*$ we define 
\begin{equation}\label{eq:tiles}
\Te_u\coloneqq f_u(\Te)\sub \Te.
\end{equation}
Note that $\Te_\emptyset=\Te$. 
Since $\Te=\pi(W)$ and $f_u(\pi(v))=\pi(uv)$ whenever $u\in W_*$ and $v\in W$ (see \eqref{pifrel}), the set $\Te_u$ consists precisely of the points  $a\in \Te$ that can be 
  represented in the form $a=\pi(w)$ with a word $w\in W$ that has
  $u$ has an initial segment. This implies that if $v\in W_*$   is a finite word  with the initial segment 
  $u\in W_*$, then $\Te_v\sub \Te_u$.  
  
  It follows from  \eqref{Tinv2}  that 
$$  \Te_u=\Te_{u1}\cup\Te_{u2}\cup \Te_{u3} $$ for each $u\in W_*$
and from \eqref{Tinv3} that    
\begin{equation}\label{eq:union1}
\displaystyle \Te=\bigcup_{u\in W_n}\Te_u 
 \end{equation}
for each $n\in \N_0$.

Since $I=[-1,1]\sub \mb{T}\sub {H}$ (as follows from Proposition~\ref{prop:CSSTgeo}) and $\diam (I) = \diam (H) =2$, we have  $\diam(\mb{T})=2.$  If $n\in \N_0$ and $u\in W_n$, then $f_u$ is a similarity map that scales  distances by the factor $1/2^n$. Hence 
\begin{equation}\label{eq:diam}
\diam(\Te_u)=2^{1-n}. 
\end{equation}

We have  $0=f_1(1)=f_2(-1)=f_3(-1)$. This implies 
\begin{equation}\label{eq:baseincl} 
0\in \Te_k= f_k(\Te)\sub f_k(H)=H_k
\end{equation}
for $k=1,2,3$. 
If $k,\ell\in\{1,2,3\}$ and $k\ne \ell$, then (see Figure~\ref{f:subtrees})
\begin{equation}\label{e:TjTk0}
{H}_k \cap {H}_\ell=\{0\}, \text{ and so } \Te_k \cap \Te_\ell=\{0\}. 
\end{equation}

\begin{figure}
 \vspace{-1cm}
 \begin{overpic}[ scale=0.7
    ]{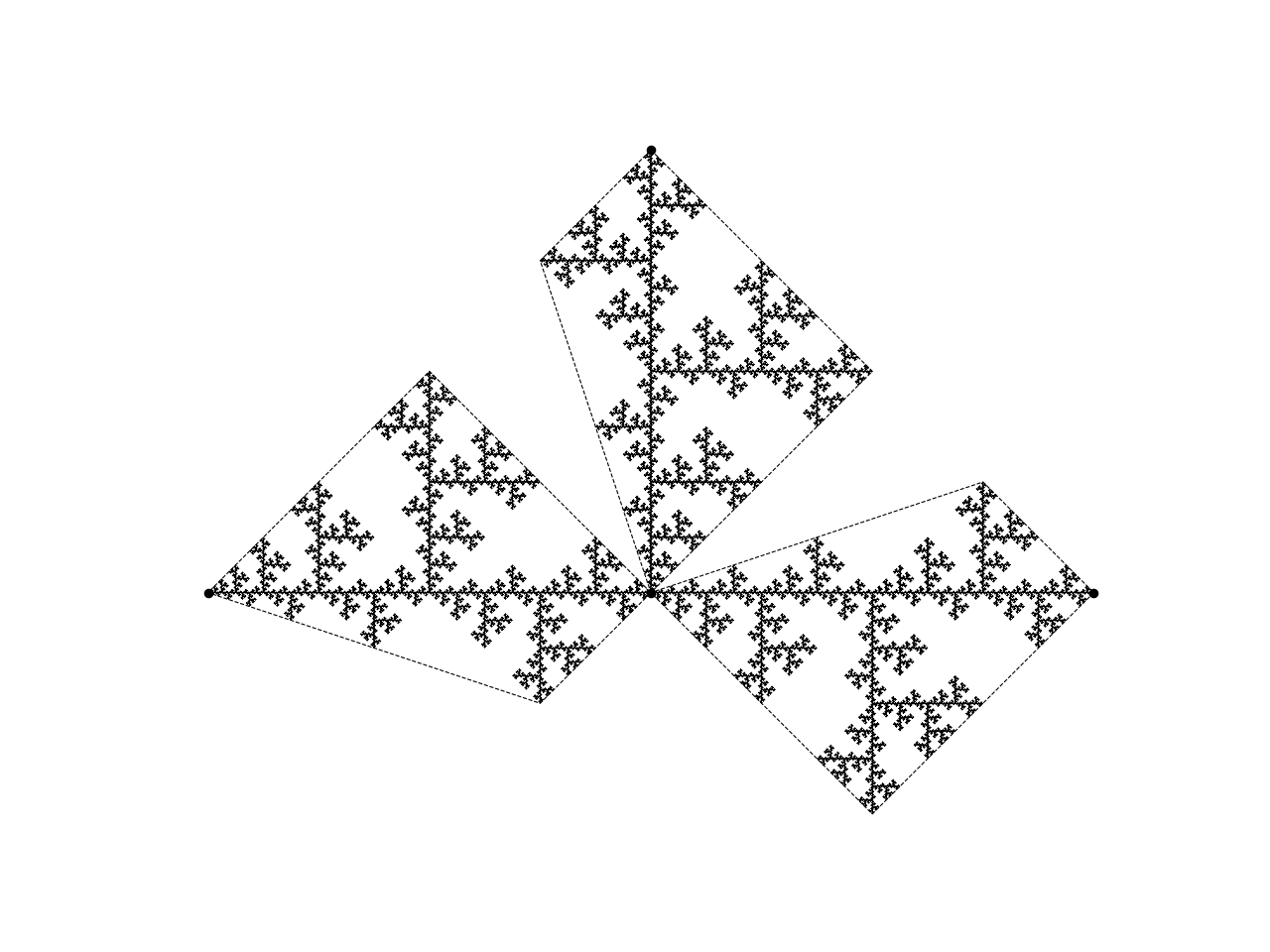}
    \put(50,65){ $i$}
    \put(14,25){ $-1$}
     \put(85,25){ $1$}
       \put(85,25){ $1$}
          \put(19, 38){ $H_1$}
            \put(78,15){ $H_2$}
              \put(63,55){ $H_3$}
            \put(50, 25){ $0$}
             \put(32, 23.5){ $\Te_1$}
                \put(67.5, 31.5){ $\Te_2$}
                  \put(45.5, 45){ $\Te_3$}
                          \end{overpic}
  \vspace{-1cm}
\caption{The CSST $\Te$ and its subtrees $\Te_1$, $\Te_2$, 
$\Te_3$.}
\label{f:subtrees}
\end{figure}

  The next lemma   provides a criterion  when two infinite words in $W$ represent the same point in $\mb{T}$ under the map $\pi$. Here we use the notation 
  $\dot{k}$ for  the infinite word $kkk\dots$ for $k\in \{1,2,3\}$.
 
     \begin{lemma}\label{l:0_triple} 
  \begin{itemize}
 \smallskip 
\item [\textnormal {(i)}] We have 
  $\pi^{-1}(0)=\{1\dot{2},2\dot{1}, 3\dot{1}\}$.

\smallskip 
\item [\textnormal {(ii)}] \  Let $v,w\in W$ with $v\ne w$. Then   
 $\pi(v)=\pi(w)$ if and only if  there exists a finite word $u\in W_{\ast}$ such that $v,w\in 
 \{u1\dot{2},u2\dot{1},u3\dot{1}\}$. In this case,
  $\pi(v)=\pi(w)=f_u(0)$. 
 \end{itemize}
\end{lemma} 
Note that if $v\in W$ and $v\in \{u1\dot{2},u2\dot{1},u3\dot{1}\}$ for some $u\in W_{\ast}$, then $u$ is uniquely determined. This and the lemma imply that each point in $\Te=\pi(W)$ has at most three peimages under the map $\pi$.

  \begin{proof} (i) Note that $1\dot{2}\in \pi^{-1}(0)$ as follows from 
  $$  f_2(1)=1 \text{ and } f_1(1)=0.$$
  Similarly, $2\dot{1}, 3\dot{1} \in \pi^{-1}(0),$ because   
  $$  f_1(-1)=-1, f_2(-1)=0, \text{ and }  f_1(-1)=-1,\, f_3(-1)=0. $$ 
  Hence $\{1\dot{2}, 2\dot{1}, 3\dot{1}\}\sub \pi^{-1}(0)$. 
  
  To prove the reverse inclusion, suppose that $\pi(w)=0$ for some $w = w_1w_2\ldots \in W$. We first consider the case $w_1=1$. 
   Then $0=f_1(a)$, where   $a\coloneqq \pi(w_2w_3\ldots)$, and so $a=1$. Since $1\in \mb{T}_2\backslash (\mb{T}_1\cup \mb{T}_3)$ as follows from \eqref{eq:baseincl}, we must have $w_2=2$. Then 
   $1=f_2(b)$, where $b\coloneqq w_3w_4\ldots$, and so $b=1 \in \mb{T}_2\backslash (\mb{T}_1\cup \mb{T}_3) $.  This  implies   $w_3=2$. Repeating the argument, we see that $2=w_2=w_3=\ldots$, and so $w=1\dot{2}$.
  
 A  very similar argument shows that if $w_1=2$, then $w=2\dot{1}$, and  if $w_1=3$, then $w=3\dot{1}$.
 
 \smallskip
(ii) Suppose that $\pi(v)=\pi(w)$ for some $u,v\in W$, $u\ne v$. Let $u\in W_{\ast}$ be the longest initial word that $v$ and $w$ have in common. So $v=uv_{n+1}v_{n+2}\dots$ and  
$w=uw_{n+1}w_{n+2}\dots$, where $n\in \N_0$ and $v_{n+1}\neq w_{n+1}$. Since $f_u$ is bijective and
 $$\pi(v)= f_u(\pi(v_{n+1}v_{n+2}\dots))=\pi(w)= f_u(\pi(w_{n+1}w_{n+2}\dots))\, , $$ we have 
$$\pi(v_{n+1}v_{n+2}\dots) = \pi(w_{n+1}w_{n+2}\dots).$$
Note that
$\pi(v_{n+1}v_{n+2}\dots)  \in \mb{T}_{v_{n+1}}$ and 
$\pi(w_{n+1}w_{n+2}\dots)  \in \mb{T}_{w_{n+1}}.$ Since 
$v_{n+1}\ne W_{n+1}$, 
by  (\ref{e:TjTk0}) this is only possible if  $\pi(v_{n+1}v_{n+2}\dots) = \pi(w_{n+1}w_{n+2}\dots)=0$. Hence 
$$v_{n+1}v_{n+2}\dots, w_{n+1}w_{n+2}\ldots\in \{1\dot{2},2\dot{1},3\dot{1}\}$$ by (i).  The  ``only if" implication  follows. Our considerations also show that  $\pi(v)=\pi(w)=f_u(0)$.
The reverse implication follows from (i). 
 \qed
\end{proof}

Our next goal is to show that $\Te$ is indeed a tree. This requires some preparation. 

 \begin{lemma}\label{lem:arc1} 
 
  \begin{itemize} 
  \smallskip 
  \item [\textnormal {(i)}]   For each  $p\in \Te$ there exists a 
  (possibly degenerate)  arc $\alpha$ in $\Te$
 with endpoints $-1$ and $p$. 

  \smallskip 
  \item [\textnormal {(ii)}] \  The sets $\Te$, $\Te\ba\{1\}$, and $\Te\ba \{-1\}$
   are arc-connected. 
      \end{itemize} 
\end{lemma}

\begin{proof} (i) Let  $p\in \Te$. Then  $p=\pi(w)$ for some $w=w_1w_2\ldots \in W$. 

 Let $v_n=w_1\dots  w_n$ and define $a_n=f_{v_n}(-1)\in \Te$ for $n\in \N_0$. Then 
$a_0=f_{\emptyset}(-1)=-1$.  
   For each $n\in \N_0$ we have 
$$ [a_n, a_{n+1})=
[f_{v_n}(-1), f_{v_n {w_{n+1}}}(-1))=
f_{v_n}\big (   [-1, f_{w_{n+1}}(-1)   )  \big ). $$
If $w_{n+1}=1$, then $f_{w_{n+1}}(-1)=f_1(-1)=-1$; so $a_n=a_{n+1}$ 
and 
$$ [a_n, a_{n+1})=\emptyset.
$$  If $w_{n+1}\in \{2,3\}$, then 
$f_{w_{n+1}}(-1)=0$; so   
$$[-1, f_{w_{n+1}}(-1))=[-1,0)\sub \Te_1\ba \{0\}, $$ and
$$ [a_n, a_{n+1}) = f_{v_n}([-1, 0))\sub f_{v_n}(\Te_{1}\ba \{0\})\sub \Te_{v_n}\sub \Te. $$ 
Moreover,
\begin{equation}\label{eq:len1} 
 \length ([a_n, a_{n+1}))=\frac{1}{2^n} \length \big([-1, f_{w_{n+1}}(-1))\big)= \begin{cases}
  2^{-n} &\text{ if } w_{n+1} = 2,3,\\
  0 &\text{ if } w_{n+1} =1. 
\end{cases}
\end{equation}
Let $$ A_n\coloneqq \{p\}\cup\bigcup_{k\ge n+1} [a_k, a_{k+1})$$
for $n\in  \N_0$. 
By what we have seen above, 
$$ [a_k, a_{k+1})\sub \Te_{v_k}\sub \Te_{v_{n+1}}$$ 
for $k\ge n+1$. Since $p=\lim_{k\to \infty}a_k$ and $\Te_{v_{n+1}}$ is closed, we also have $p\in \Te_{v_{n+1}}$, and so 
$$ A_n \sub \Te_{v_{n+1}}. $$ 
This implies that 
$$[a_n, a_{n+1})\cap A_n=\emptyset $$
for each  $n\in \N_0$. 
Indeed, if $w_{n+1}=1$  this is clear, because then $[a_n, a_{n+1})=\emptyset$.

If $w_{n+1}=2$, then
$$ A_n\sub  \Te_{v_{n+1}}=f_{v_n}(f_2(\Te))=f_{v_n}(\Te_2),$$
which  implies that 
$$ [a_n, a_{n+1})\cap A_n \sub f_{v_n}(\Te_1\ba\{0\})\cap f_{v_n}(\Te_2)= f_{v_n}( (\Te_1\ba\{0\})\cap \Te_2)=\emptyset. $$ 
If $w_{n+1}=3$, then $[a_n, a_{n+1})\cap A_n=\emptyset$  by the same reasoning.
This shows that the sets 
$$ [a_0,a_1), [a_1, a_2), [a_2, a_3),\dots, \{p\}$$ are 
pairwise disjoint. As $n\to \infty$,  we have   $a_n\to p$ and also $\diam(A_n)\to 0$ by \eqref{eq:len1}. 
 Therefore, the union 
  \begin{equation}\label{eq:alp}
   \alpha =   [a_0,a_1)\cup  [a_1, a_2)\cup [a_2, a_3)\cup \dots \cup  \{p\}
   \end{equation} 
  is an arc in $\Te$ joining $a_0=-1$ and $p$ (if $p=-1$, this arc is degenerate). We have proved (i).

 To prepare the proof of (ii),  we claim that if $p\ne 1$, then this arc $\alpha$ does not contain
  $1$.  
Otherwise, we must have $1\in [a_n, a_{n+1})\sub \Te_{v_n} $ for some $n\in \N_0$.
This shows that $1$ can be written in the form $1=\pi(u),$
where $u\in W $ is an infinite word starting with the finite word $v\coloneqq v_n$ (note that this  and the statements below are trivially true for $n=0$).
On the other hand, we have $f_2(1)=1$ which implies that 
$1=\pi(\dot 2)$. 
By Lemma~\ref{l:0_triple}~(ii) this is only possible if  all the letters in $v$ are 
$2$'s. Then $f_{v}(1)=1$ and it follows that 
$$ 1=f_v(1)\in [a_n,a_{n+1})= f_{v}
\big([-1, f_{w_{n+1}}(-1))\big).$$ Since $f_v$ is a bijection, this implies  that $1\in [-1, f_{w_{n+1}}(-1))$. Now 
$f_{w_{n+1}}(-1)\in \{-1, 0\}$, and we obtain a contradiction. So indeed, $1\not \in \alpha$.

\smallskip 
(ii) Let $p,q\in \Te$ with $p\ne q$ be arbitrary. In order to show 
that $\Te$ is arc-connected, we have to find   an arc  $\ga$ in
$\Te$ joining $p$ and $q$. Now by the construction in (i) we can find arcs $\alpha$ and $\beta$ in $\Te$ joining $p$ and $q$ to $-1$, respectively. 
Then the  desired arc $\ga$ can be found in the union $\alpha \cup \beta$  as follows. Starting from $p$, we travel long $\alpha $ until we first hit  $\beta$, say in a point $x$. Such a  point $x$ exists, because $-1\in  \alpha\cap \beta\ne \emptyset$. Let $\alpha'$ be the (possibly degenerate) subarc of $\alpha $ with endpoints $p$ and $x$, and $\beta'$ be the subarc of $\beta$ with endpoints $x$ and $q$. Then $\ga =\alpha'\cup \beta'$ is an arc   in $\Te$ joining $p$ and $q$. 

The arc-connectedness of $\Te\ba \{1\}$ is proved by the same argument. Indeed, if $p,q\in \Te\ba \{1\}$, then  by the remark in the last part of the proof of (i), the arcs  $\alpha$ and $\beta$ constructed as in (i)  do not contain $1$. Then the arc $\ga\sub \alpha \cup \beta$ does not contain $1$ either. 

Finally, to show that  $\Te\ba \{-1\}$ is arc-connected, we assume that  $p,q\in \Te\ba \{-1\}$.
 If $x$ is, as  above, the first point on $\beta$ as we travel  along $\alpha $ starting from $p$, then it suffices to show that $x\ne -1$, because then $-1\not \in \ga$. This in turn will follow if we can show that $\alpha$ and $\beta$ have another point in common besides
  $-1$.  

To find such a point, we revisit the above construction. 
Pick  $w=w_1w_2\ldots\in W$ and  $u=u_1u_2\ldots \in W$ such that $p=\pi(w)$ and $q=\pi(u)$. Let $\alpha $ and $\beta$ be the arcs for $p$ and $q$, respectively, as constructed in (i). Then  $\alpha$ is as in \eqref{eq:alp} and we can write the other arc $\beta$ 
 as 
$$  \beta =   [b_0,b_1)\cup  [b_1, b_2)\cup [b_2, b_3)\cup \dots \cup  \{q\}, $$
where $b_n=f_{u_1\dots u_n}(-1)$ for $n\in \N_0$. Since $p\ne q$, we have $w\ne u$, and so there exists a largest  $n\in \N_0$ such that 
$v\coloneqq w_1\dots w_n=u_1\dots u_n$ and $w_{n+1}\ne u_{n+1}$. Then $a_n=b_n=f_v(-1)\in \alpha\cap \beta$. If $a_n= b_n\ne -1$, we are done. So we may assume that 
$a_n=b_n=f_v(-1)=-1$. Then 
$a_0=\dots =a_n=-1$, and so $w_k=u_k=1$ for $k=1, \dots, n$.  This shows that  all letters in  $v$ are equal to $1$.

Since the letters $w_{n+1}$ and $u_{n+1}$ are distinct,
one of them is different from $1$. We may assume $u_{n+1}\ne 1$.
Then $f_{u_{n+1}}(-1)=0$,
and so 
$(b_n, b_{n+1}]=f_v((-1, 0]) \sub \beta\setminus \{-1\}$. 
Here we used that $f_v$ is a homeomorphism with $f_v(-1)=-1$. 

 Since $p=\pi(w) \ne -1=\pi(\dot 1)$, we have  $w\ne \dot 1$ and so there exists a smallest $\ell\in \N$ such that $w_{n+\ell}\ne 1$.
Then  $f_{w_{n+\ell}}(-1)=0$ and so a simple computation using 
$w_{n+1}=\dots= w_{n+\ell-1}=1$ shows  that
$$ c\coloneqq f_{w_{n+1}\dots w_{n+\ell}}(-1)= 
f_{w_{n+1}\dots w_{n+\ell-1}}(0)= 2^{1-\ell}-1\in (-1, 0].$$ 
Hence 
$$a_{n+\ell}=f_v(c)\in f_v((-1, 0])\sub \beta\setminus \{-1\}.$$ 
It follows that $a_{n+\ell}\in \alpha\cap  \beta$ and $a_{n+\ell}\ne -1$ as desired.  \qed
\end{proof}

The next lemma will help us to identify the branch points of $\Te$ once we know that $\Te$ is a tree.  

  \begin{lemma}\label{lem:trip}
  \begin{itemize} 
  \smallskip 
  \item [\textnormal {(i)}] 
The  components of $\Te\ba\{0\}$ are given by 
the non-empty sets $\Te_1\ba\{0\}$, $\Te_2\ba\{0\}$, $\Te_3\ba\{0\}$. 
 \smallskip 
  \item [\textnormal {(ii)}] \ If $u\in W_*$, then $\Te\ba\{f_u(0)\}$ has exactly three  components. The 
 sets $\mb{T}_{u1}\ba\{f_u(0)\}$,
 $\mb{T}_{u2}\ba\{f_u(0)\}$,
 $\mb{T}_{u3}\ba\{f_u(0)\}$ are each contained  in a different component of $\mb{T}\ba\{f_u(0)\}$. 
\end{itemize} 
  \end{lemma}
In the proof we will use the following general facts about 
 components of a subset $M$ of   a metric space $X$. 
Recall that a set $A\sub M$ is relatively closed in $M$ if $A=\overline A\cap M$, or equivalently, if each limit point of $A$ that belongs to $M$ also belongs to $A$. Each component $A$ of $M$ is relatively closed in $M$, because its relative closure $\overline A \cap M$ is a connected subset of $M$ with $A\sub\overline A \cap M$. Hence $A=\overline A \cap M$, because $A$ is a  component of $M$ and hence a maximal connected subset of $M$. 
 
  If $A_1, \dots, A_n\sub M$  for some $n\in \N$ are non-empty, pairwise disjoint, relatively closed, and connected sets with $M=A_1\cup \dots \cup A_n$, then these sets are the components of $M$.

  \begin{proof}
(i) Each of the sets $\mb{T}\backslash\{1\}$ and 
 $\mb{T}\backslash\{-1\}$ is  non-empty,  and connected by 
Lemma~\ref{lem:arc1}~(ii).   Therefore, the sets 
$$\mb{T}_1\backslash\{0\}=f_1(\mb{T}\backslash\{1\}),\  \mb{T}_2\backslash\{0\}=f_2(\mb{T}\backslash\{-1\}),\  \mb{T}_3\backslash\{0\}=f_3(\mb{T}\backslash\{-1\})$$ are non-empty and   connected.
They are also relatively closed in $\mb{T}\backslash\{0\}$ and  pairwise disjoint by \eqref{e:TjTk0}. Since $\Te=\Te_1\cup\Te_2\cup \Te_3$ we have 
$$ \Te\ba\{0\} = (\Te_1\ba\{0\})\cup (\Te_2\ba\{0\}) \cup (\Te_3\ba\{0\}).$$ This implies that the  sets $\Te_k\ba\{0\}$, $k=1,2,3$,  are the  components of 
$\mb{T}\backslash\{0\}$.  The statement follows. 

\smallskip 
(ii) We prove this by induction on the length $n\in \N_0$ of the word 
$u\in W_*$. If $n=0$ and so $u=\emptyset$,  this follows from  statement (i).

Suppose the statement is true for all words of length $n-1$, where $n\in \N$. Let $u=u_1\dots u_n\in W_n$ be an arbitrary word of length $n$. We set $\ell\coloneqq u_1$ and $u'\coloneqq u_2\dots u_{n}$. Then $u =\ell u'$. To be specific and ease notation, we will assume that $\ell=1$. The other cases $\ell=2$ or $\ell=3$ are completely analogous and we will skip the details.

Note that $f_u(0)\ne 0$. Indeed, if 
$$0=f_u(0)=f_u(\pi(1\dot2))=
\pi(u1\dot2), $$ then $u1\dot2\in \{1\dot2, 2\dot 1, 3\dot1\}$ by Lemma~\ref{l:0_triple}~(i). This is only possible $u1=1$. This is a contradiction, because $u$ has length $n\ge 1$. Hence $f_u(0)\ne 0$. Since $u_1=\ell=1$, we  have $f_u(0)\in \Te_1\ba\{0\}$. 

By induction hypothesis, $\Te\ba \{f_{u'}(0)\}$ has exactly three connected components $V_1$, $V_2$, $V_3$, and
we may assume that 
$ \Te_{u'k}\ba  \{f_{u'}(0)\} \sub V_k$ 
for $k=1,2,3$. 
It follows that 
$$f_\ell(\Te\ba \{f_{u'}(0)\})=f_1(\Te\ba \{f_{u'}(0)\})= \Te_1\ba \{f_{u}(0)\}$$ 
has exactly three connected components $U_k=f_1(V_k)\sub \Te_1$ with 
$$\Te_{uk}\ba\{f_{u}(0)\} =\Te_{1u'k}\ba  \{f_{1u'}(0)\} =
 f_1(\Te_{u'k}\ba  \{f_{u'}(0)\}) \sub 
f_1( V_k)=U_k$$
for $k=1,2,3$.

Let  $k\in \{1,2,3\}$. Then we have  $V_k=\overline {V_k}\cap \Te\ba \{f_{u'}(0)\}$, because  $V_k$ is a component of  
$\Te\ba \{f_{u'}(0)\}$ and hence relatively closed in 
$\Te\ba \{f_{u'}(0)\}$.   This implies that 
\begin{align*}
U_k&=f_1(V_k)=f_1( \overline {V_k}\cap \Te\ba \{f_{u'}(0)\})=
f_1( \overline {V_k})\cap \Te_1\ba \{f_{u}(0)\}\\
&=\overline {f_1(V_k)}\cap \Te_1\ba \{f_{u}(0)\} = \overline {U_k} 
\cap  \Te_1\ba \{f_{u}(0)\}.
\end{align*}   
Since $\Te_1\sub \Te$ is compact, $U_k\sub \Te_1$, and so $\overline {U_k} \sub \Te_1$,  this shows that every limit point of $U_k$ distinct from $f_u(0)$ belongs to $U_k$. Hence $U_k$ is relatively closed in $\Te\ba  \{f_{u}(0)\}$. 

Exactly one of the components of $\Te_1\ba \{f_{u}(0)\}$, say $U_1$, contains the point 
$0\in \Te_1\ba \{f_{u}(0)\}$. Then $U'_1\coloneqq U_1\cup \Te_2\cup \Te_3$ is a relatively closed subset of $\Te \ba\{f_u(0)\}$. This set is also connected, because the sets $U_1$, $\Te_2=f_2(\Te)$, $\Te_3=f_3(\Te)$  are connected and have the point $0$ in common.  Hence  the connected sets $U'_1$, $U_2$, $U_3$  are pairwise disjoint, relatively closed in 
$ \Te \ba  \{f_{u}(0)\}$, and 
$$ \Te \ba  \{f_{u}(0)\}=  (\Te_1 \ba \ \{f_{u}(0)\})\cup \Te_2\cup \Te_3=U'_1\cup U_2\cup U_3. $$ This implies that 
 $\Te \ba  \{f_{u}(0)\}$ has exactly the three connected 
 components $U'_1$, $U_2$, $U_3$. 
Moreover,   $\Te_{u1}\ba\{f_{u}(0)\}$, $\Te_{u2}\ba\{f_{u}(0)\}$, $\Te_{u3}\ba\{f_{u}(0)\}$ lie in the different components 
$U_1'$, $U_2$, $U_3$ of $\Te \ba  \{f_{u}(0)\}$, respectively. 
This provides the inductive step, and the statement follows.  \qed
 \end{proof}
 
 We can now show that $\Te$ is a metric tree. 
 
%
%
  
  
  \begin{proofof} {\em Proof of Proposition~\ref{prop:CSST}.} We know that $\Te$ is  compact, contains at least two points, and is 
 arc-connected  by Lemma~\ref{lem:arc1}.   
 
 Let $p\in \Te$ and $n\in \N$ be arbitrary, and define 
 $$ N=\bigcup \{ \Te_u: u\in W_n \text{ and } p\in \Te_u\}. $$ 
Since each of the sets 
   $\Te_u=f_u(\Te)$, $u\in W_*$, is a compact and connected
   subset of $\Te$,  
 the set  $N$ is  connected. Moreover, since each of the finitely many sets  $\Te_u$, $u\in W_n$, is closed, we can find 
 $\delta>0$ such that 
 $$ \dist(p, \Te_u)\ge \delta$$ 
 whenever $u\in W_n$ and  $p\not\in \Te_u$. Then we have $B(p, \delta)\cap \Te \sub N$ by \eqref{eq:union1}, and so $N$ is a connected relative neighborhood of $p$ in $\Te$. It follows from  \eqref{eq:diam} 
 that  $\diam(N)\le 2^{2-n} $. This shows that each point in $\Te$ has \arbly\ small connected  neighborhoods in $\Te$. Hence $\Te$ is locally connected.

To complete the proof, it remains to show that the arc joining two given distinct points in $\Te$  is unique. 
For this we argue by contradiction and assume that there are two distinct arcs in $\mb{T}$ with the same  endpoints. By considering suitable subarcs of these arcs, we can reduce
 to the following situation: there are arcs $\alpha,\beta\sub \mb{T}$ that have the distinct   endpoints $a,b \in \mb{T}$ in common, but no other points. 
 
 To see that this leads to a contradiction, we represent the points 
 $a$ and $b$ by words in $W$; so $a=\pi(v)$ and
 $b=\pi(w)$, where $v=v_1v_2\ldots$ and $w=w_1w_2\ldots$ are in $W$. Since $a\ne b$ and every point in $\mb{T}$ has at most three such representations by Lemma~\ref{l:0_triple}~(ii), we can find a pair $v$ and $w$ representing $a$ and $b$ with the largest common initial word, say $v_1=w_1,\ldots, v_n=w_n$, and $v_{n+1}\ne w_{n+1}$ for some maximal $n\in \N_0$.  
 
 Let $u=v_1\ldots v_n=w_1\dots w_n$ and 
 $$t=f_u(0)=\pi(u1\dot2)=\pi(u2\dot1)=\pi(u3\dot1). $$ 
 Then $t\ne a,b$. To see this, assume that $t=a$, say. We have  
 $w_{n+1} \in \{1,2,3\}$, and so,  say $w_{n+1}=1$.
 But then 
 $a=t=\pi(u1\dot2)$ and $b=\pi(u1w_{n+2}\ldots)$. So $a$ and $b$ 
 are represented by words with the common initial segment $u1$ that is longer than $u$. This contradicts the choice of $v$ and $w$. 
 The cases $w_{n+1} =2$ or $w_{n+1}=3$ lead to a contradiction in a similar way.

 So indeed $t=f_u(0)\ne a,b$. Moreover $a=\pi (uv_{n+1}\ldots)\in 
 \mb{T}_{uv_{n+1}}\ba\{t\}$ and similarly 
 $b\in \mb{T}_{uw_{n+1}}\ba\{t\}$. Since $v_{n+1}\ne w_{n+1}$ the points $a$ and $b$ lie in different components of 
 $\mb{T}\ba\{t\}$ by Lemma~\ref{lem:trip}~(ii). So any arc joining $a$ and $b$ must pass through $t$. Hence $t\in \alpha\cap\beta$, but $t\ne a,b$. This contradicts our assumption that the arcs $ \alpha$ and $\beta$ have no other points
 than their endpoints $a$ and $b$ in common.   \qed
\end{proofof}

If $M\sub \Te$, then we denote by $\partial M\sub \Te$ the relative boundary of $M$ in $\Te$. 

 \begin{lemma}\label{p:decomp}
Let $n\in \N$ and $u\in W_n$. Then 
\begin{equation} \label{eq:markpts} 
\partial \Te_u\sub \{f_u(-1), f_u(1)\}. 
\end{equation} 
Moreover, if  $p\in \partial \Te_u$, 
then $p=f_w(0)$ for some word $w\in W_*$ of length $\le n-1$.
\end{lemma}

In particular, the set $\partial \Te_u$ contains at most two points.

\begin{proof} We prove this by induction on $n$.  First consider $n=1$. So let $u=k\in W_1=\{1,2,3\}$.   
 Then $\Te_k\ba\{0\}$ is a  component $\Te\ba\{0\}$  by Lemma~\ref{lem:trip}~(i).
 Hence Proposition~\ref{prop:CSST} and Lemma~\ref{lem:vartree}~(i) imply that 
 $\Te_k\ba\{0\}$ is a relatively  open set in $\Te$. So each of its points lies in the relative interior of  $\Te_k$ and cannot lie in $\partial \Te_k$.  Therefore,  $\partial \Te_k\sub \{0\}$. Since \begin{equation} 
0= f_\emptyset(0)=f_1(1)=f_2(-1)=f_3(-1), 
\end{equation}
 the statement is true for $n=1$. 

Suppose the statement is true for all words in $W_n$, where 
$n\in \N$. Let  $u\in W_{n+1}$ be arbitrary. Then $u=vk$, where $v\in W_n$ and $k\in \{1,2,3\}$.  By what we have just seen, the set  $\Te_k\ba \{0\}$ is open in $\Te$. Hence 
$$f_v( \Te_k\ba \{0\})=f_u(\Te)\ba\{f_v(0)\}=\Te_u\ba\{f_v(0)\}$$ is a relatively open subset  of $f_v(\Te)=\Te_v$. So if $p\in \Te_u$ is 
not an interior point of $\Te_u$ in $\Te$, then $p=f_v(0)$ or $p$ is not an interior point of $\Te_v$ in $\Te$ and hence belongs to the boundary of $\Te_v$.
This and the induction hypothesis imply that 
$$ \partial \Te_u\sub \{f_v(0)\}\cup \partial \Te_v\sub 
\{f_v(0), f_v(-1), f_v(1)\}. $$ 
From this we conclude that  each point $p\in \partial \Te_u\sub \{f_v(0)\}\cup \partial \Te_v$  
can be 
written in the form $f_w(0)$ for an appropriate word $w$ of length $\le n$. This is clear if  $p=f_v(0)$ and follows for $p \in \partial \Te_v$
 from the induction hypothesis. 

Now $\Te_u=f_u(\Te)$ is compact and so closed in $\Te$. Hence 
$\partial \Te_u\sub \Te_u$. On the other hand, $\Te_u$ contains only two of the points  $f_v(0)$, $f_v(-1)$, $f_v(1)$. Indeed, if 
$k=1$, then  $1\not\in \Te_1\sub H_1$, and so 
$f_v(1)\not\in f_v(\Te_1)=\Te_u$. It follows that $ \partial \Te_u\sub \{ f_v(-1), f_v(0)\}. $
Note that $f_1(-1)=-1$ and $f_1(1)=0$, and so 
$$f_v(-1)=f_v(f_1(-1))=f_u(-1) \text{ and }
 f_v(0)=f_v(f_1(1))=f_u(1).  $$ 
Hence $$\partial \Te_u\sub \{ f_u(-1), f_u(1)\}.$$

Very similar considerations show that if $k=2$, then 
$$\partial \Te_u\sub \{f_v(0), f_v(1)\}=\{f_u(-1),f_u(1)\},$$ and if $k=3$, then 
$$\partial \Te_u\sub \{f_v(0)\}=\{f_u(-1)\}.$$ The statement follows. 
 \qed
\end{proof}

The next lemma shows that all branch points of $\mb{T}$ are of the form $f_u(0)$ with  $u\in W_{\ast}$.

\begin{lemma} \label{lem:branchCSST} The branch points of $\Te$ are exactly the points of the form  $t=f_u(0)$ for   some finite word $u\in W_{\ast}$. They are 
triple points of $\Te$. 
\end{lemma}

\begin{proof} By Lemma~\ref{lem:trip}~(ii) we know that each point 
$t=f_u(0)$ with $u\in W_{\ast}$ is a triple point of the tree $\Te$.
We have to show that there are no other branch points of $\Te$. 
 
 So suppose that $t$ is a branch point of $\mb{T}$, but $t\ne f_u(0)$ for each  $u\in W_{\ast}$. Then we can find (at least) three distinct components $U_1$, $U_2$, $U_3$ of $\mb{T}\ba\{t\}$.  Pick a point $x_k\in U_k$  and  choose  $n\in \N$ such that 
$ |x_k-t|> 2^{1-n}$ 
for $k=1,2,3$. By \eqref{eq:union1} we can find $u\in W_n$ such that $t\in \mb{T}_u$. 
Then $t$ is distinct from the points in the relative  boundary $\partial \Te_u$, because they have the form $f_w(0)$ for some $w\in W_\ast$
 (see Lemma~\ref{p:decomp}). Hence $t$ is contained in the relative interior of $\mb{T}_u$ in $\mb{T}$. 
Moreover,  $\diam(\mb{T}_u)=2^{1-n}$, and so $x_k\not \in \mb{T}_u$. For $k=1,2,3$ let $\alpha_k$ be the arc in $\Te$ joining $x_k$ and $t$. As we travel from $x_k$ to $t$ along $\alpha_k$, there exists a first point $y_k \in \mb{T}_u$. Then $y_k\in \partial \Te_u$ and so $y_k\ne t$. Let   $\beta_k$ be the subarc of $\alpha_k$  with endpoints  $x_k$ and $y_k$. Then $\beta_k$ is a connected set in  $\Te\ba\{t\}$. Since $x_k\in \beta_k$, it follows that $\beta_k\sub U_k$, and so $y_k\in U_k$. 

This shows that  the points $y_1, y_2, y_3$ are distinct and contained in the relative boundary $\partial \Te_u$. This is impossible,  because by  Lemma \ref{p:decomp} the set $\partial \Te_u$ consists of at most two points.  \qed
\end{proof}

We can now prove Proposition~\ref{prop:T123} which shows that 
  $\Te$ satisfies the conditions in Theorem~\ref{criterion} and  belongs to the class of trees $\T_{3}$. 

\begin{proofof}{\em Proof of Proposition~\ref{prop:T123}.}
 By Lemma~\ref{lem:branchCSST} each branch point of $\Te$ is a triple point and each set $\Te_u$ for  $u\in W_n$ and $n\in \N$ contains the triple point $t=f_u(0)$. The sets $\Te_u$, $u\in W_n$, 
cover $\Te$ and have small diameter for $n$ large. It follows that the triple points are dense in $\Te$.  \qed
\end{proofof}

In order to show that  $\Te$ is a quasi-convex subset of 
$\C$, we first require a lemma. 

\begin{lemma}\label{lem:0dist}
There exists a constant $K>0$ such that if $p\in \Te$ and  $\alpha$ is the arc in $\Te$ joining $0$ and $p$, then  
\begin{equation}\label{lem:len2}
 \length(\alpha) \le K|p|.
\end{equation}
\end{lemma}

In particular, the arc $\alpha$ is a rectifiable curve.

\begin{proof} Let $p\in \Te$ be arbitrary. We may assume that 
$p\ne 0$. Then $p=\pi(w)$ for some $w=w_1w_2\ldots \in W$. 
For simplicity we assume $w_1=3$. The other cases, $w_1=1$ and $w_1=2$, are very similar and we will only present the details for 
$w_1=3$.

 Since $p\ne 0=\pi(3\dot 1)$, we have 
$w_2w_3\ldots \ne \dot 1$. Hence there exists a smallest number 
  $n\in \N$ such that 
such that $w_{n+1}\ne 1$. Let $v=w_1\dots w_{n}$ be the initial word of $w$ and $w'=w_{n+1}w_{n+2}\ldots $ be the tail of $w$. The word $v$ has  the form $v=31\dots 1$, where the sequence  of  $1$'s could possibly be empty.
Note that $q\coloneqq \pi(w')\in \Te_{w_{n+1}}\sub \Te_2\cup\Te_3\sub H_2\cup
H_3$. 
Since  
$$ c_0\coloneqq \dist(-1, H_2\cup H_3)>0$$(see Figure~\ref{f:Hsets}), for the distance of $q$ and $-1$ we  have 
$|q+1|\ge c_0$. We also have  $f_v(q)=p$, and $f_v(-1)=0$, because 
$f_1(-1)=-1$ and $f_3(-1)=0$. It follows that 
\begin{equation}\label{eq:p} 
 |p|=|f_v(q)-f_v(-1)|=\frac 1{2^n}  |q+1|\ge \frac {c_0}{2^n}. 
\end{equation}

Now  define $a_0=0=f_v(-1)$ and  $a_k= 
f_{vw_{n+1}\dots w_{n+k-1}}(0)$ for $k\in \N$ (here $w_{n+1}\dots w_{n+k-1}=\emptyset$ for $k=1$). Note that then 
$$ a_1=f_v(0)=f_{w_1\dots w_n}(0)=f_{31\dots 1}(0)=
f_3(2^{1-n}-1)=i/2^n, $$ 
and so 
$$[a_0, a_1]= [f_v(-1), f_{v}(0)]= [0, i/2^{n}]\sub [0,i]
\sub \Te.$$ 
This also shows that $\length ([a_0, a_1])= 1/2^n. $

For  $k\in \N$ we have $f_{w_{n+k}}(0)\in \{-1/2, 1/2, i/2\}$, and 
$[0, f_{w_{n+k}}(0)]\sub \Te$.
This implies that 
$$[a_k, a_{k+1}]= f_{vw_{n+1}\dots w_{n+k-1}}\big([0, f_{w_{n+k}}(0)]\big) \sub \Te$$ and
$\length  ([a_k, a_{k+1}])= 1/{2^{n+k}}$ for  $k\in \N$. 
Since $\lim_{k\to \infty}a_k=\pi(w)=p$, we can   concatenate the intervals $[a_k, a_{k+1}]\sub \Te$ for $k\in \N_0$, add the endpoint 
$p$, and obtain  a path $\ga$ in $\Te$ that joins $0$ and $p$  with  
$$ \length(\ga)= \sum_{k=0}^\infty \frac1{2^{n+k}}=\frac1{2^{n-1}}.$$ 
The (image of the) path $\ga$ will contain the unique arc $\alpha$ 
in $\Te$ joining $0$ and $p$ and so
$\length(\alpha)\le 1/{2^{n-1}}$. If we combine this with \eqref{eq:p},  then inequality \eqref{lem:0dist} follows with $K=2/c_0$.  \qed
\end{proof}


We can now  show  that $\Te$ is indeed a quasi-convex subset of $\C$.

\begin{proofof}{\em Proof of Proposition~\ref{prop:quasicon}.}
Let $a,b\in \Te$ be arbitrary. We may assume that 
$a\ne b$. Then there are words $u=u_1u_2\ldots\in W$ and 
$v=v_1v_2\ldots \in W$ such that $a=\pi(u)$ and $b=\pi(v)$. Since 
$a\ne b$, we have $u\ne v$ and so there exists a smallest number $n\in \N_0$ such that $u_1=v_1, \dots, u_n=v_n$ and $u_{n+1}\ne v_{n+1}$. Let $w=u_1\dots u_n=v_1\dots v_n$,  
$u'=u_{n+1}u_{n+2}\ldots \in W$ and $v'=v_{n+1}v_{n+2}\ldots \in W$. We define $a'=\pi(u')$ and $b'=\pi(v')$. 
Set $k= u_{n+1}$ and $\ell=v_{n+1}$. Then $k\ne \ell$, $a'\in \Te_k\sub H_k$, and $b'\in  \Te_\ell\sub H_\ell$. We now use the following
elementary geometric estimate: there exists a constant $c_1>0$ such that 
$$ |x-y| \ge c_1(|x|+|y|), $$ 
whenever $x\in H_k$, $y\in H_\ell$, $k,\ell\in \{1,2,3\}$, $k\ne \ell$. 
Essentially, this follows from the fact that the sets $H_1$, $H_2$, $H_3$ are contained in closed sectors in $\C$ that are pairwise disjoint except for the common point $0$. 

In our situation, this means that 
$$ |a'-b'|\ge c_1 (|a'|+|b'|). $$ 
Let $\sigma$ and $\tau$ be the arcs in $\Te$ joining $0$ to $a'$ and $b'$, respectively. 
Then  $\sigma\cup \tau$ contains the arc $\alpha'$  in $\Te$ joining $a'$ and $b'$. Then it follows from Lemma~\ref{lem:0dist} that 
\begin{equation}\label{eq:prelim}
\length (\alpha')\le \length (\sigma)+\length (\tau)\le K(|a'|+|b'|)
\le L |a'-b'|
\end{equation}
with $L \coloneqq  K/c_1$.  

For the  similarity $f_w$ we have  $f_w(a')=a$ and $f_w(b')=b$. 
Since $f_w(\Te)\sub \Te$, it follows that $\alpha\coloneqq f_w(\alpha')$ is the unique arc in $\Te$ joining $a$ and $b$.
Since $f_w$ scales distances by a fixed factor (namely $1/2^n$), 
\eqref{eq:prelim} implies the desired inequality 
$ \length(\alpha)\le  L |a-b|. $  \qed
 \end{proofof}  
 
As we already discussed in the introduction, by Proposition~\ref{prop:quasicon} we can define a new metric $\varrho$ on $\Te$ 
by setting 
\begin{equation}\label{def:rho}
 \varrho(a,b)=\length(\alpha) 
 \end{equation} 
for $a,b\in \Te$,  
where $\alpha$ is the unique arc in $\Te$ joining  $a$ and $b$. 
Then the metric space $(\Te, \varrho)$ is  geodesic, and we have 
$$|a-b|\le \varrho(a,b)\le L|a-b|$$
for $a,b\in \Te$, where $L$ is the constant in Proposition~\ref{prop:quasicon}. This implies that   the metric spaces $\Te$ (as equipped with the Euclidean metric) and $(\Te, \varrho)$ are bi-Lipschitz equivalent by the identity map. 

We  now want to reconcile Definition~\ref{def:CSST} with the construction of the CSST as an abstract metric space  outlined in the introduction.  We require an auxiliary statement.

\begin{lemma}\label{lem:pairwisedist}
Let $n\in \N_0$. Then the sets 
\begin{equation}\label{eq:setsfw}
 f_w(\Te\ba\{-1\}), \ w\in W_n,
 \end{equation}
are pairwise disjoint and their union is equal to $\Te\ba\{-1\}$. 
\end{lemma}

\begin{proof} This is proved by induction on $n\in \N_0$. For $n=0$ the statement is clear, because then 
$f_\emptyset(\Te\ba\{-1\})=\Te\ba\{-1\}$ is the only set in \eqref{eq:setsfw}. 

Suppose the statement is true for some $n\in \N$. Then for each 
$u\in W_n$ the sets 
\begin{align*} 
f_{u1}(\Te\ba\{-1\})&= f_u(\Te_1\ba\{-1\}),\\
f_{u2}(\Te\ba\{-1\})&= f_u(\Te_2\ba\{0\}),\\
f_{u3}(\Te\ba\{-1\})&= f_u(\Te_3\ba\{0\})
\end{align*}
 provide a decomposition of $f_u(\Te\ba\{-1\})$ into three pairwise disjoint subsets 
as follows from \eqref{eq:union1} for $n=1$, \eqref{eq:baseincl},  and \eqref{e:TjTk0}.
This and the induction hypothesis imply that the sets
$ f_{uk}(\Te\ba\{-1\})$, $u\in W_n$, $k\in \{1,2,3\}$, and hence the sets   $f_w(\Te\ba\{-1\})$, $w\in W_{n+1}$, are
pairwise disjoint, and their union is equal to $\Te\ba\{-1\}$.  This is the inductive  step, and the statement follows. \qed
\end{proof}

We now consider the sets $J_n$, $n\in \N_0$, 
as in Proposition~\ref{prop:CSSTgeo}. Here $J_0=I=[-1, 1]$ is a line segment of length $2$.  Since $(-1,1]\sub \Te\ba\{-1\}$, the previous lemma implies that for each $n\in \N_0$, 
the sets $f_w((-1,1])$, $w\in W_n$, are pairwise disjoint half-open line segments of length $2^{1-n}$. The union of the closures 
$f_w([-1,1])=f_w(I)$, $w\in W_n$, of these line segments is the set  $J_n$. In particular, $J_n$ consists of $3^n$ line segments of length $2^{1-n}$ with  pairwise disjoint interiors. 

Note that for $w\in W_n$ we have  
\begin{align*}  f_{w1}((-1,1])\cup  f_{w2}((-1,1])\cup f_{w3}((-1,1]) &=f_w((-1,0])\cup f_w((0,1])
\cup f_w((0,i])\\
&=  f_{w}((-1,1])\cup  f_w([0,i]). 
\end{align*} 
An induction argument based on this   shows that for $n\in \N_0$ we have a decomposition 
\begin{equation}\label{eq:decompJn} 
J_n\ba\{-1\} =\bigcup_{w\in W_n} f_w((-1,1]) 
\end{equation}
of $J_n\ba\{-1\}$ into  the pairwise disjoint sets   $f_w((-1,1])$, $w\in W_n$.

 In the passage from  $J_n$ to $J_{n+1}$ we can think of each line segment  $f_w(I)=f_w([-1,1])$  as being replaced with  
$$  f_{w1}(I)\cup  f_{w2}(I) \cup f_{w3}(I) =f_w([-1,0])\cup f_w([0,1])
\cup f_w([0,i]). $$
So  $f_w([-1,1])$ is split into two intervals $f_w([-1,0])$ and 
$f_w([0,1])$, and  at its midpoint $f_w(0)$ a  new interval 
$f_w([0,i])$ is ``glued"  to $f_w(0)$. This is exactly the procedure described in the introduction. Note  that Lemma~\ref{lem:pairwisedist} implies that these new intervals 
$f_w([0,i])\sub f_w(\Te\ba\{-1\})$, $w\in W_n$, are pairwise disjoint. Moreover, each such interval  $f_w([0,i])$ meets the set $J_n$ only in the point $f_w(0)$ and in no other point of $J_n$.   Indeed, by 
\eqref{eq:decompJn} and Lemma~\ref{lem:pairwisedist} we have 
\begin{align*} 
f_w((0,i])\cap J_n &= f_{w3}((-1,1])\cap J_n =   
f_{w3}((-1,1])\cap J_n\ba\{-1\}\\
&=  f_{w}((0,i])\cap \bigcup_{u\in W_n} f_u((-1, 1])\\
\sub  \big(f_{w}((0,i])&\cap  f_w((-1, 1])\big)\cup 
 \bigcup_{u\in W_n,\, u\ne w}   f_{w}(\Te\ba\{-1\})\cap 
 f_u(\Te\ba\{-1\})=\emptyset. 
\end{align*}

 It is clear that $J_n$ is compact, and one can show by induction based on the replacement procedure just described  that $J_n$ is connected. Hence each $J_n$ is a subtree of $\Te$ by Lemma~\ref{lem:subtree}. The  metric $\varrho$ in \eqref{def:rho} restricted to  $J_n$, $n\in \N_0$, and to  $J\coloneqq \bigcup_{n\in N_0} J_n$  is just  the natural Euclidean path metric on  these sets. In particular, 
$\varrho$  is a geodesic metric on $J$. These considerations imply that $(J_n, \varrho)$ for $n\in \N$, and hence  $(J, \varrho)$, are isometric to the abstract versions of these spaces defined in the introduction. 

By Proposition~\ref{prop:CSSTgeo} the tree $\Te$ is the equal 
to  closure $\overline J$ in $\C$. Since on $ J$ the Euclidean metric and the metric $\varrho$ are comparable, 
the set  $\Te=\overline J$ is homeomorphic to the space obtained from the completion of the geodesic metric metric space $(J, \varrho)$. This is how we described the CSST as an abstract metric space in the introduction.

%
%
%

 \section{Decomposing  trees in $\T_m$} \label{s:proof}
   
 In the previous   section we have seen that for each $n\in \N$ the CSST admits a de\-com\-position   
$$\Te=\bigcup_{u\in W_n} \Te_u$$ 
into subtrees. We will now consider an \arb\ tree in $\T_m$, $m\in \N$, $m\ge 3$,  
and find  similar decompositions into subtrees. Our goal is to have decompositions for each level $n\in \N$ so that the conditions  
(i)--(iii) in  Proposition~\ref{nested} are satisfied. 

Note that each tree  class $\mc{T}_{m}$  is non-empty. 
Namely, for each $m\in \N$, $m\ge 3$, a tree in $\T_m$ can be obtained by essentially the same method as for the construction of the CSST as an abstract metric space outlined in the introduction. The only difference is that instead of gluing one line segment  of length $2^{-n}$
to the midpoint $c_s$ of a line segment $s$ of length $2^{1-n}$ obtained in the $n$th step, we glue  endpoints of  
$m-2$ such segments  to $c_s$.  Since from a purely logical point of view we will not need the fact that
$\mc{T}_{m}$ is non-empty for the proof of Theorem~\ref{t:infinite},  
 we will skip further  details.

We now fix $m\in \N$, $m\ge 3$,  for the rest of this section. We consider the alphabet $\mathcal{A}=\{1,2, \dots, m\}$. In the following, words will  contain only letters in this fixed alphabet and we use the simplified notation for the sets of words
$W$, $W_n$, $W_*$  as discussed in Section~\ref{s:CSST}.

Let $T$ be an arbitrary  tree in the class  $\T_{m}$.  We will now  define subtrees $T_u$ of $T$ for all levels $n\in \N$ and all $u\in W_n$. The boundary $\partial T_u$ of $T_u$ in $T$ will consist of one or two points that are leaves of $T_u$ and branch points of $T$. We  consider each point  in  $\partial T_u$ as a  {\em marked leaf} in $T_u$ and will assign to it an appropriate 
sign  $-$ or $+$ so that if there are two marked leaves in $T_u$, then they carry different  signs.  Accordingly,  we refer to the  points in $\partial T_u$ as the {\em signed marked leaves} of $T_u$.  The same point may carry different signs in different subtrees.  We write $p^-$ if a marked leaf $p$ of $T_u$  carries the sign $-$ and $p^+$ if it 
carries the sign $+$. To refer to this sign, we also write $\sgn(p, T_u)=-$ in the first and $\sgn(p, T_u)=+$ in the second case. 
If  $T_u$ has exactly one marked leaf, we call $T_u$ a  {\em leaf-tile} and if there are two marked leaves an {\em arc-tile}.

The reason why we want to use these markings is that it will help us to consistently label the  subtrees so that if another tree $S$ in $\T_m$ is decomposed by the same  procedure, then we obtain 
decompositions of our trees $T$ and $S$ into  subtrees on all levels $n$ that satisfy the analogs of \eqref{eq:iff1}  and \eqref{eq:iff2} (here $u\in W_n$ will play the role of the index $i$ on each level $n$). While
 \eqref{eq:iff1} is fairly straightforward to obtain,   \eqref{eq:iff2} 
 requires a more careful approach and this is where the markings will help us (see Lemma~\ref{lem:condtrue2}~(ii) and its proof).

 For the construction we will use an inductive procedure on $n$. 
 As in Section~\ref{sec:toptree} (see  \eqref{eq:height} and the discussion before Lemma~\ref{finite}), 
 for each branch point $p\in T$, we let $H_T(p)$ be its height, i.e.,  the diameter of the third largest  branch of $p$ in $T$. If $\delta>0$, then by Lemma~\ref{finite} there are only finitely many branch points $p$ of 
 $T$ with  height  $H_T(p)>\delta$, and in particular there is one for which this quantity is maximal. 
 
For the first step $n=1$, we  choose a branch   point $c$ of $T$  with maximal height $H_T(c)$.  Since $T$ is in the class $\T_m$, this branch  point $c$ has $m=\nu_T(c)$  branches in $T$. So we can enumerate the distinct branches by the letters in our alphabet 
 as    $T_k$, $k\in \mathcal{A}$.  
 
We choose $c$ as the signed marked leaf in each  $T_k$,   
 where we set $\sgn(c, T_1)=+$ and $\sgn(c, T_k)=-$ for $k\ne 1$. So the set of signed marked leaves is $\{c^+\}$ in $T_1$ and $\{c^-\}$ in $T_k$, $k\ne 1$. 
 Note that $\partial T_k=\{c\}$ as follows from Lemma~\ref{lem:vartree}~(ii) and that $c$ is indeed a leaf in $T_k$ by  Lemma~\ref{lem:branch}  for each $k\in \mathcal{A}$. 

 Suppose that for some $n\in \N$ and all $u\in W_n$ we have constructed 
   subtrees $T_u$ of $T$ such that $\partial T_u$ consists of one or two signed marked leaves of $T_u$  that are  branch points of $T$. We will now construct the subtrees of the $(n+1)$-th 
   level as follows
 by subdivision of the trees $T_u$.

 Fix  $u\in W_n$. To decompose $T_u$ into subtrees, we will 
 use a suitable branch point $c$ of $T$ in $T_u\ba\partial T_u$. The choice of  $c$ depends on whether $\partial T_u$ contains one or two elements, that is, whether $T_u$ is a leaf-tile or an arc-tile. We will explain this precisely below, but first record some facts that are true in both cases. 
 
 Since $c\in T_u\ba\partial T_u$ is an interior point of $T_u$, 
there is a bijective correspondence between the branches of $c$ in $T$ and in $T_u$ (see Lemma~\ref{lem:subbranches}). So $\nu_{T_u}(c)=\nu_T(c)=m$, and we can label the distinct branches 
 of $c$ in $T_u$  by $T_{uk}$, $k\in \mathcal{A}$. We will 
  choose these labels depending on the signed marked leaves of $T_u$. Among other things,  if $T_u$ has a marked leaf $p^-$, then $p$ is passed to 
 $T_{u1}$ with the same sign. Similarly, a marked leaf $p^+$ of $T_u$
 is passed to $T_{u2}$ with the same sign. We will momentarily explain this in more detail (see the {\em Summary} below).  
 
  In any case, we   have 
\begin{equation}\label{eq:piecedecomp}
T_u=\bigcup_{k\in  \mathcal{A}} T_{uk}.
\end{equation}
Each set $T_{uk}$ is a subtree of $T_u$ and hence also of $T$. We call these subtrees the {\em children} of $T_u$ and $T_u$ the {\em parent} of its children.  Note that two distinct children of $T_u$ have only the point $c$ in common and no other points. 

Before we say more about the precise labelings of the children of $T_u$ and their signed leaves, we first want to identify the boundary of each child;  namely, we want to show that 
\begin{equation}\label{eq:bound}
 \partial T_{uk}=  \{c\}\cup   (\partial T_u\cap T_{uk})
\end{equation}
for each $k\in \mathcal{A}$.

To see this, first note that  $T_{uk}$ is a subtree of $T$.  Hence 
$T_{uk}$  contains all of its boundary points and so $\partial T_{uk}\sub T_{uk}$.  We have $c\in \partial T_{uk}$, because $c\in T_{uk}$ and  every neighborhood of $c$ contains points in the complement of $T_{uk}$ as follows from Lemma~\ref{lem:vartree}~(ii) (here it is important that there are at least two branches of $c$). If $p\in T_{uk}\sub T_u$ and 
$p\not\in \{c\}\cup \partial T_u$, then a sufficiently small neighborhood  $N$ of $p$ belongs to $ T_u$. Since $T_{uk}\ba\{c\}$ is relatively open in $T_u$ (this follows from Lemma~\ref{lem:vartree}~(i)), we can shrink this neighborhood so that $p\in N\sub T_{uk}$. So no point $p$  in $T_{uk}$ can be a boundary point of $\partial T_{uk}$ unless it belongs  to $\{c\}\cup \partial T_u$. It follows that  $\partial T_{uk}\sub \{c\}\cup 
(\partial T_u\cap T_{uk})$.

On the other hand,  we know that $c\in \partial T_{uk}$. If $p\in 
\partial T_u \cap T_{uk}$, then $p$ is  a boundary point of $T_{uk}$, because every neighborhood of  $p$ contains elements in the complement of $T_u$ and hence in the complement of $T_{uk}\sub T_u$. This gives the other inclusion in \eqref{eq:bound}, and 
\eqref{eq:bound} follows.

The identity \eqref{eq:bound}  implies that each point in  $\partial T_{uk}$ is a branch point of $T$, because $c$ is and  the points in $\partial T_u$ are also  branch points of $T$ by construction on the previous level $n$. Moreover, each point $p\in \partial T_{uk}\sub T_{uk}$  is a leaf of $T_{uk}$, because if $p=c$, then $p$  is a leaf in $T_{uk}$ by Lemma~\ref{lem:branch}. Otherwise, $p\in \partial T_u$. Then $p$ is a leaf of $T_u$ by construction and hence a leaf of $T_{uk}$ by the discussion after  Lemma~\ref{lem:subbranches}.

For the  choice of the branch point $c\in T_u\ba\partial T_u$, the precise labeling of the children $T_{uk}$, and the choice of the signs of the leaves of $T_{uk}$ in $\partial T_{uk}$, we now consider two cases for the set $\partial T_u$. See Figure \ref{decof} for an illustration. 
 
 \smallskip
 {\em Case 1:} $\partial T_u$ contains precisely one element, say $\partial T_u=\{a\}$.  Note that 
 $T_u$ is a subtree of $T$ and so an infinite set. So $T_u\ba\partial T_u\ne \emptyset$. All points in $T_u\ba\partial T_u$  
 are interior points of  $T_u$. Since branch points 
 in $T$ are dense (here we use that $T$ belongs to $\T_m$), there exist branch points of $T$ in 
 $T_u\ba\partial T_u$.  We choose a  branch point $c\in T_u\ba\partial T_u $ with maximal height $H_T(c)$ among all such  branch points.  This is possible by Lemma~\ref{finite}.

Since $a\in \partial T_u\sub T_u\ba\{c\}$, precisely one of the children of $T_u$ contains $a$. 
We now consider two subcases  depending on the sign of the marked leaf $a$.     

If $\sgn(a,T_u)=-$,  then we  choose a labeling of the children  so that $a\in T_{u1}$. 
It then follows from \eqref{eq:bound} that  $\partial T_{u1}=\{a,c\}$ and $\partial T_{uk}=\{c\}$
for $k\ne 1$. We choose signs so that the set of signed marked leaves is $\{a^-, c^+\}$ in $T_{u1}$ and $\{c^-\}$ in $T_{uk}$, $k\ne 1$.

If $\sgn(a,T_u)=+$,   then we choose a labeling such that 
 $a\in T_{u2}$. Then again by \eqref{eq:bound} we have $\partial T_{u2}=\{a,c\}$ and 
 $\partial T_{uk}=\{c\}$ for $k\ne 2$.  
  We choose signs so that the set of signed marked leaves is 
$\{c^+\}$ in $T_{u1}$, $\{c^-, a^+\}$ in $T_{u2}$, and 
 $\{c^-\}$ in $T_{uk}$, $k\ne 1,2$.  

\begin{figure}

\begin{center}
\includegraphics [scale=0.3, 
]{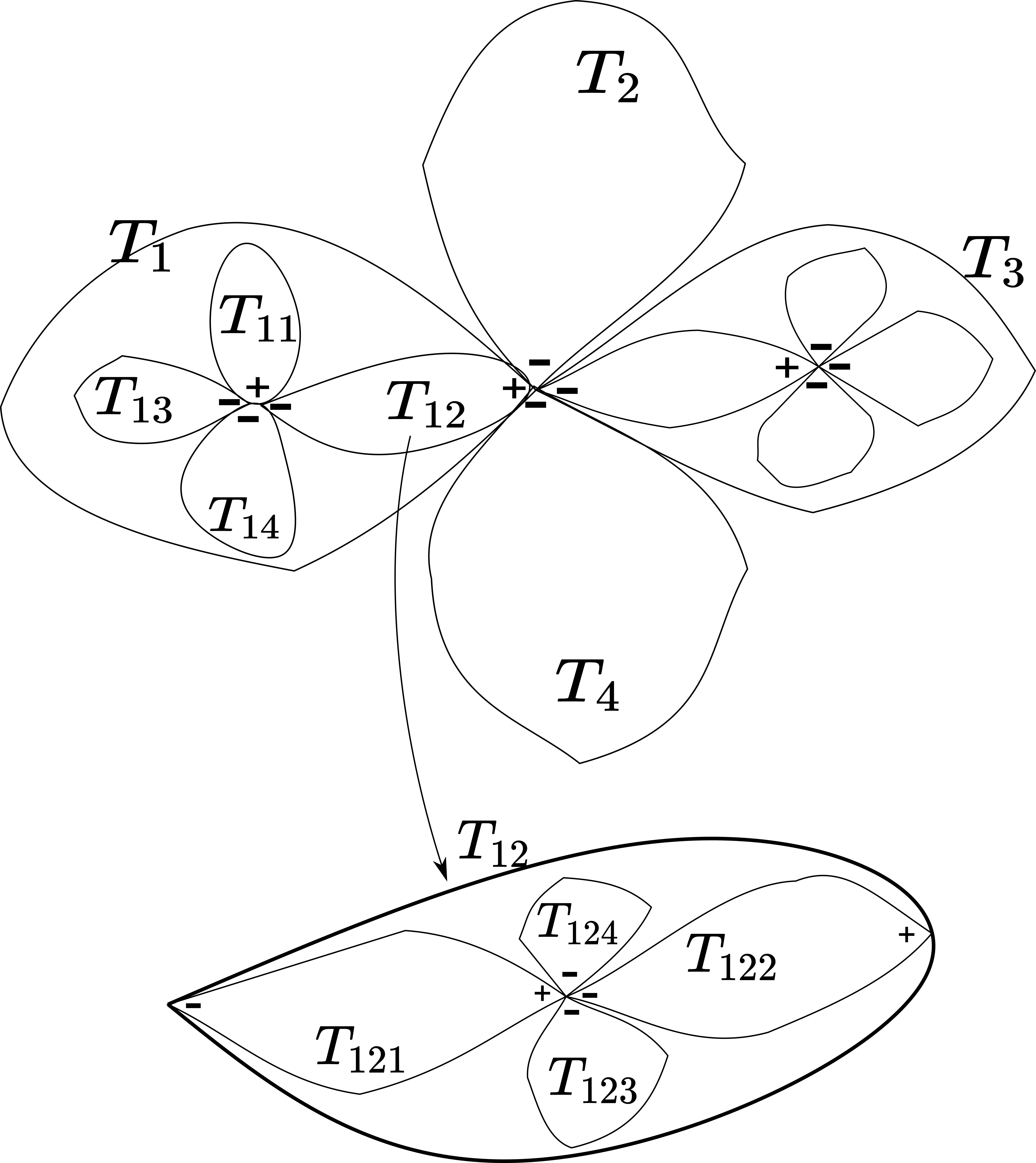}
\caption{An illustration for the decomposition of  subtrees with one marked leaf (top) or  two marked leaves (bottom).}
\label{decof}
\end{center}

\end{figure}

\smallskip
  {\em Case 2.} $\partial T_u$ contains precisely two  elements, say $\partial T_u =\{a^-, b^+\}$.  Then  we choose a branch point  $c\in (a, b)$ of $T$ 
 such that it has the maximal  height $H_T(c)$ among all branch  points  that lie on  $(a, b)$. The existence of $c$ is guaranteed by Lemma~\ref{triple_dense} and Lemma~\ref{finite}.
 Note that $(a,b)\sub T_u$, because $T_u$ is a subtree of $T$.
  
The points  $a$ and $b$ lie in different branches of $c$ in $T_u$ as  follows from Lemma~\ref{lem:vartree}~(iii). We choose the labels for the children of $T_u$  so that  $a\in T_{u1}$ and 
 $b\in T_{u2}$. Then by \eqref{eq:bound} we have 
 $\partial T_{u1}=\{a,c\}$, $\partial T_{u2}=\{c,b\}$, and 
  $\partial T_{uk}=\{c\}$, $k\ne 1,2$. We choose  signs  so that   
  the set of marked leaves is 
 $\{a^-, c^+\}$ in $T_{u1}$, $\{c^-, b^+\}$ in $T_{u2}$, and $\{c^-\}$, in  $T_{uk}$ for $k\ne 1,2$.

 \smallskip
 The most important points of our contruction can be summarized as follows.
 
 \smallskip
   {\em Summary:} $T_{uk}$ is a  subtree of $T$ such that $\partial T_{uk}$ consists of  one or two points. These  points are branch points of $T$ and leaves of $T_{uk}$. Moreover, the signs of the points in each set $\partial T_{uk}$ are chosen so that these signs differ if $\partial T_{uk}$
 contains two points. If $c$ is the branch point used to decompose 
 $T_u$, then $c$ is a marked leaf in all the children of $T_u$, namely the marked leaf $c^+$ in $T_{u1}$ and $c^-$ in $T_{uk}$ for $k\ne 1$. 
 
If $T_u$ has a marked leaf $p^-$, then $p$ is passed to the child 
 $T_{u1}$ with the same sign. Similarly, a marked leaf $p^+$ of $T_u$
 is passed to $T_{u2}$ with the same sign. So marked leaves are  passed to a unique child and they retain their signs.
 \smallskip

  Since Cases 1 and 2 exhaust all possibilities, this  completes the inductive step in the construction of the trees on level $n+1$ and their marked leaves.  So we obtain subtrees $T_u$ 
 of $T$ for all $u\in W_*$. Here it is convenient to set $T_\emptyset=T$ with an empty set of marked leaves.  
 
 
If one  applies our procedure  to choose 
 signs for the points in $\partial \Te_u$ for the subtrees $\Te_u$ 
 of the CSST defined in Section~\ref{s:CSST}, then one can recover these signs directly  by a simple rule without going through the recursive process. Namely, by Lemma~\ref{p:decomp} we have  
 $\partial \Te_u\sub \{f_u(-1), f(1)\}$. Then it is not hard to see that 
 for  
 $p\in \partial \Te_u$, we have  $\sgn(\Te_u, p)=+$ if $p=f_u(1)$ and 
 $\sgn(\Te_u, p)=-$ if $p=f_u(-1)$.
 
 We now summarize some facts about the subtrees $T_u$ of $T$ that we just defined.

\begin{lemma}\label{lem:condtrue}  The following statements are true: 

\begin{itemize}
 
\smallskip 
\item [\textnormal {(i)}] $\displaystyle T=\bigcup_{u\in W_n} T_u$ 
for each $n\in \N$.   

\smallskip
\item[\textnormal {(ii)}]  \ If   $n\in \N$, $u,v\in W_n$, $u\ne v$, 
and  $T_u\cap T_v\ne \emptyset $, then $T_u\cap T_v$ consists of precisely one point $p\in T$, which is a marked leaf  in both $T_u$ and $T_v$.  

\smallskip 
\item [\textnormal {(iii)}] \ \, For  $n\in \N$, $u\in W_n$, and $v\in W_{n+1}$, we have $T_v\sub T_u$ if and only if $u=vk$ for some 
$k\in\mathcal {A}$. 

\smallskip 
\item [\textnormal {(iv)}] \  \, For each $u\in W_*$ let $c_u$ be the branch point chosen in the decomposition of $T_u$ into children. Then $c_u\ne c_v$ for all $u,v\in W_*$ with  $u\ne v$.  
\end{itemize}
\end{lemma} 

\begin{proof} (i) This immediately follows from \eqref{eq:piecedecomp} and induction  on $n$. 

\smallskip
 (ii) We prove this by induction on $n$. By choice of the subtrees 
$T_k$ for  $k\in \mathcal{A}=W_1$  and their marked leaves this is clear for $n=1$.

Suppose the statement is  true for all distinct words  of 
length $n-1$, where $n\ge 2$.  Now consider two words $u,v\in W_n$ of length $n$  with $u\ne v$ and $T_u\cap T_v\ne \emptyset$.   Then $u=u'k$ and $v=v'\ell$,  where  $u',v'\in W_{n-1}$ and $k,\ell\in \mathcal{A}$. 

If  $u'=v'$, then $T_u$ and $T_v$ are two of the branches 
 obtained from $T_{u'}$ and a suitable branch point  $c\in T_{u'}$. In this case, $\{c\}= T_u\cap T_v$ and $c$ is a marked leaf in both $T_u$ and $T_v$. 
 
In the other case,  $u'\ne v'$. Then 
$T_{u'}\cap  T_{v'}\ne \emptyset$, because 
$T_{u}\cap  T_{v}\ne \emptyset$,  $T_u\sub T_{u'}$, and $T_v\sub T_{v'}$. By induction hypothesis, 
$T_{u'}\cap  T_{v'}$ consists of  precisely one point $p$, which  is  a marked leaf in both $T_{u'}$ and $T_{v'}$. 
Then necessarily $T_u\cap T_v=\{p\}$.
Moreover, $p$ is a marked leaf in both $T_u$ and 
$T_v$, because marked leaves are passed to children. 
The  statement follows. 

\smallskip
(iii)  Let $n\in \N$ and $u\in W_n$. Then we have $T_{uk}\sub T_u$ for each $k\in \mathcal{A}$ by our construction.  Conversely, suppose $T_v\sub T_u$, where $v=v'k\in W_{n+1}$ with  $v'\in W_{n}$ and $k\in \mathcal{A}$. Then $T_{v'}\cap T_u\supseteq T_v$ contains more than one point. By (iii)
this implies that $v'=u$. The statement follows. 

\smallskip
(iv) If $u\in W_n$, $n\in \N_0$, then by construction $c_u\in T_u$  does not lie in the set $\partial T_u$ of marked leaves of $T_u$. By (ii) this implies that $c_u\not\in T_{w}$ for each $w\in W_n$, $w\ne u$. It follows that the points $c_u$, $u\in W_n$, are all distinct, and 
none of them is contained  in the union of sets $\partial T_u$, $u\in W_n$. By our construction this union is equal to the set of all  points $c_v$, where   $v\in W_*$ is a word of length $\le n-1$. This shows that the branch points 
 $c_u$, $u\in W_n$, used to define the subtrees of level $n+1$ are all distinct and distinct from any of the previously chosen branch points for  levels $\le n$. The statement follows from this.  \qed
\end{proof}

\begin{lemma} \label{lem:diamto0} 
We have 
$ \displaystyle \lim_{n\to \infty} \sup\{ \diam(T_u): u\in W_n\} =0. $ 
\end{lemma}
\begin{proof} Let $\delta_n\coloneqq  \sup \{ \diam(T_u): u\in W_n\}$
for $n\in \N$. It is clear that the sequence 
$\{ \delta_n\}$ is non-increasing. To show that $\delta_n\to 0$ as $n\to \infty$, we argue by contradiction. Then there exists 
$\delta>0$ such that $\delta_n\ge \delta$ for all $n\in \N$. 
This means that for each $n\in \N$ there exists $u\in W_n$ with 
\begin{equation}\label{eq:del}
\diam(T_u)\ge \delta.
\end{equation}
We now use \eqref{eq:del} to find an infinite word $w=w_1w_2\ldots\in W$ such that 
\begin{equation}\label{eq:del2}
 \diam(T_{w_1\dots w_n})\ge \delta
 \end{equation}
  for all $n\in \N$. 
The word $w$ is constructed inductively as follows. 
 One of the finitely many letters $k\in \mathcal{A}$  must have the property that there are arbitrarily long words $u$ starting with $k$ such that \eqref{eq:del} is true.  

We define $w_1=k$. Note that then $\diam(T_{w_1})\ge \delta$. 
T  By choice of $w_1$,   one of the letters  $\ell\in  \mathcal{A}$  must have the property that 
there are arbitrarily long words $u$ starting with $w_1\ell$  
such that \eqref{eq:del} is true. We define $w_2=\ell $.
Then $\diam(T_{w_1w_2})\ge \delta$.  Continuing in this manner, we can find $w=w_1w_2\ldots\in W$ satisfying  \eqref{eq:del2}. 

Obviously, 
$$ T_{w_1}\supseteq T_{w_1w_2} \supseteq T_{w_1w_2w_3} 
 \supseteq\dots\, .$$ 
So the subtrees $K_n=T_{w_1\dots w_n}$, $n\in \N$, of $T$ form a descending family of compact sets with $\diam(K_n)\ge \delta$. 
This implies that 
$$ K=\bigcap_{n\in \N} K_n$$ is a non-empty compact subset of 
$T$ with $\diam(K)\ge \delta.$

In particular, we can choose $p,q\in K$ with $p\ne q$. Then 
$p,q\in K_n$ for each $n\in \N$. Since $K_n$ is a subtree of $T$, we then have $[p,q]\sub K_n$. Moreover, by Lemma~\ref{triple_dense} there exists a branch  point $x$ of $T$ contained  in $(p,q)\sub K_n$. 
By Lemma~\ref{lem:vartree}~(iii) the points $p$ and $q$ lie in different components of  $K_n\ba\{x\}$.  In particular, for each $n\in \N$ the point $x$ is not a leaf of $K_n$ and hence distinct from the marked leaves of $K_n$.

 
 By Lemma~\ref{finite} there are only  finitely many  branch  points $y_1, \dots, y_s$ of $T$ distinct   from $x$ with $H_T(y_j)\ge H_T(x)>0$ for $j=1, \dots, s$. This implies that at most $s$ of the trees $K_n$ are leaf-tiles, i.e., have only one marked leaf. Indeed, 
 if $K_{n}$  an leaf-tile, then it is decomposed into  branches  by use of a branch point $c\in K_n\ba \partial K_n$ with the largest height $H_T(c)$. The point $c$ is then a marked leaf in each of the 
 children of $K_n$ and in particular in $K_{n+1}$.  
 Since the branch  point $x\in K_{n}$ is distinct from the marked leaves of $ K_{n}$ and $K_{n+1}$, we have $x\in K_n\ba\partial K_n$
 and $x\ne c$.  So  $x$ was not chosen to decompose $K_{n}$, and we must have 
 $H_T(c)\ge H_T(x)$. Since the branch  points $c$ that appear from leaf-tiles at different levels $n$ are all distinct as follows from Lemma~\ref{lem:condtrue}~(iv), we can have at most $s$ 
 leaf-tiles in the sequence  $K_n$, $n\in \N$. 
This implies  that there exists $N\in \N$ such that $K_n$ for $n\ge N$ is an arc-tile and so has precisely two marked leaves.  
 
 Let $a,b\in K_N$ with $a\ne b$ be the marked leaves of $K_N$.
 As we travel from $x$ along $[x,a]\sub K_N$ towards $a$, there is a first point $x'$ on $[a,b]$. Then $x'\ne a$.  Otherwise,  
$x'=a$. Then 
 $[x,a]$ and $[a,b]$ have only the point $a$ in common, which implies that $[x,a]\cup [a,b]$  is an arc equal to $[x,b]$. Then $a\in (x,b)$, which by Lemma~\ref{lem:vartree}~(iii) implies that $x,b\in K_N$ lie in different components of $K_N\ba\{a\}$. This contradicts  
the fact that $a$ is a leaf of $K_N$ and so $K_N\ba\{a\}$ has only one component. Similarly, one can show that 
$x'\ne b$.  

The point $x'$ is a branch  point of $T$. This is clear if $x'=x$.
If  $x'\ne x$, this follows from Lemma~\ref{lem:tripcrit}, because $a,b,x\ne x'$ and the arcs 
$[a,x')$, $[b,x')$, $[x,x')$ are pairwise disjoint. 

The tree $K_{N+1}$ is a branch of $K_N$ obtained from a branch point $c\in (a,b)$ of $T$ with largest height $H_T(c)$ among all branch  points  on $(a,b)$. We have $x'\ne c$. Otherwise, $x'=c$. Then $x\ne x'$, because $x'=c$ is a marked leaf of $K_{N+1}$ and 
$x$ is distinct from all the marked leaves 
in any of the sets $K_n$. This implies that  
the points $a,b,x$ lie in different components  of $K_N\ba\{x'\}$
and hence in different branches of $x'$ in $K_N$.  Since  $a$ and $b$ are the marked leaves of $K_N$, the  branches containing $a$ and $b$ are arc-tiles and all other branches of $x'=c$ in $K_N$ are leaf-tiles. The unique branch of $x'$ in $K_N$   containing $x$, which is equal to $K_{N+1}$, must be a leaf-tile by the way we  decomposed  $T$.  This is impossible by  choice of $N$  and so indeed $x'\ne c$. 
Note that this implies  $H_T(c)\ge H_T(x')$. 

Since $x'\ne c$, $c\in (a,b)$, and $[x,x')\cap [a,b]=\emptyset$, we have 
$[x,x']\sub K_N\ba\{c\}$. So $x'$ lies in the same branch of $c$ in $K_N$ as $x$, which is $K_{N+1}$. Moreover, depending on whether $c\in (a,b)$ lies on the right  or  left of $x'\in (a,b)$, we have $x'\in (a,c)$ or 
$x'\in (c,b)$. In the first case, $[a,c]\sub K_{N+1}$ and $a$ and $c$ are the marked leaves of $K_{N+1}$. In the second case,  
$[c,b]\sub K_{N+1}$ and $c$ and $b$ are the  marked leaves of $K_{N+1}$. So in both cases, if  $a'$ and $b'$ are the marked leaves of $K_{N+1}$, then $x'\in (a',b')$, $[x,x']\sub K_{N+1}$, and 
$[x,x')\cap [a',b']=\emptyset$. 

These facts allow us 
to repeat the argument for $K_{N+1}$ instead of $K_N$. 
Again $K_{N+1}$ is decomposed into branches by choice of a branch point $c'\in (a',b')$. We must have $c'\ne x'$, because otherwise we again obtain a contradiction to the fact that $K_{N+2}$ is not a leaf-tile. This implies that $H_T(c')\ge H_T(x')$. Continuing in this manner, we obtain an infinite sequence of  branch points $c,c', \dots$. By construction these branch points are all distinct and have a height $\ge H_T(x')$. This is impossible by Lemma~\ref{finite}. We obtain a contradiction that establishes the statement.    \qed
\end{proof}

The previous argument shows that each branch  point $x$  of $T$ will eventually be chosen as a branch  point in   the decomposition  of $T$ into the subtrees $T_u$, $u\in W_*$. Indeed, otherwise $x$ is distinct from all the marked leaves of any of the subtrees $T_u$, $u\in W_*$. 
This in turn implies that there exists a unique infinite word 
$w=w_1w_2\ldots \in W$ such that $x\in K_n\coloneqq T_{w_1\dots w_n}$ for $n\in \N$. From this one obtains a contradiction as in the last part of the proof of Lemma~\ref{lem:diamto0}.

\begin{lemma}\label{lem:condtrue2}  Let $m\in \N$, $m\ge 3$, and  suppose $T$ and $S$ are trees in $\mc{T}_m$. Assume that  subtrees $T_u$ of $T$ and $S_u$ of $S$ with signed marked leaves have been defined for $u\in W_*$ by the procedure described above. Then the following statements are true:

\begin{itemize}
 
\smallskip 
\item [\textnormal {(i)}] Let $n\in \N$, $u\in W_n$, and $v\in W_{n+1}$. Then $T_v\sub T_u$ if and only if $S_v\sub S_u$. 

\smallskip
\item[\textnormal {(ii)}]  \ For  $n\in \N$ and  $u,v\in W_n$ with $u\ne v$ we have  $T_u\cap T_v\ne \emptyset $
if and only if 
$S_u\cap  S_v\ne \emptyset$. Moreover, if these intersections are non-empty, then they are singleton sets, say 
$\{p\}=T_u\cap T_v$ and $ \{\widetilde p\}=S_u\cap  S_v $.
The point  $p$ is a signed marked leaf in $T_u$ and $T_v$, the point 
$\widetilde p$ is a signed marked leaf in $S_u$ and $S_v$, 
$\sgn(p, T_u)=\sgn(\widetilde p, S_u)$,  and $\sgn(p, T_v)=\sgn(\widetilde p, S_v)$. 
\end{itemize}
\end{lemma} 

In (ii) we are actually only interested in the statement that 
$T_u\cap T_v\ne \emptyset $
if and only if 
$S_u\cap  S_v\ne \emptyset$. The additional claim in (ii) will help us to prove this statement by an  induction  argument.

\begin{proof} (i) This follows from Lemma~\ref{lem:condtrue}~(iii) applied to the decompositions of $T$ and $S$. Indeed, we have 
$T_v\sub T_u$ if and only if $v=uk$ for some $k\in \mathcal{A}$ if and only if $S_v\sub S_u$. 

\smallskip
(ii)  We prove this  by induction on $n\in \N$. The case 
  $n=1$ is clear by how the decompositions were chosen.

  Suppose the claim is  true for  words  of 
length $n-1$, where $n\ge 2$.  Now consider two words $u,v\in W_n$ of length $n$  with $u\ne v$.  Then $u=u'k$ and $v=v'\ell\in W_{n}$,  where  $u',v'\in W_{n-1}$ and $k,\ell\in\mathcal{A}$. Since the claim is symmetric in $T$ and $S$, we may assume that $T_u\cap T_v\ne \emptyset$. 

If  $u'=v'$, then $T_u$ and $T_v$ are two of the branches obtained from $T_{u'}$ and  a branch  point $c\in T_{u'}$. In this case, 
$T_u\cap T_v=\{c\}$ and $c$ is a marked leaf in both $T_u$ and $T_v$. 
 Similarly, $S_u$ and $S_v$ are two of the branches obtained from $S_{u'}$ and  a branch  point $\widetilde c\in S_{u'}$. We  have $S_u\cap S_v=\{\widetilde c\} $ and $\widetilde c$ is a marked leaf in both $S_u$ and $S_v$.
 Moreover, $c$  has the same sign in  $T_u$ as $\widetilde c$ in $S_u$. Indeed, by the  choice of labeling in the decomposition, this sign is $+$ if $k=1$ and $-$ otherwise. Similarly, 
 $c$ has the same sign in $T_v$ as $\widetilde c$ in $S_v$. 
 This shows that the statement is true in this case.

In the other case,  $u'\ne v'$. Then 
$T_{u'}\cap  T_{v'}\ne \emptyset$, because
$T_{u}\cap  T_{v}\ne \emptyset$,  $T_u\sub T_{u'}$, and $T_v\sub T_{v'}$. Then by induction hypothesis, 
$T_{u'}\cap  T_{v'}$ consists of precisely one point $p$ that is  a marked leaf in  both $T_{u'}$ and  $T_{v'}$. The set  $S_{u'}\cap S_{v'}$ consists of one point $\widetilde p$ that is a marked leaf in $S_{u'}$ and $S_{v'}$. Moreover, we  have 
$\sgn(p,T_{u'})=\sgn(\widetilde p,S_{u'})$ and $\sgn(p,T_{v'})=\sgn( \widetilde p,S_{v'})$. Since $\emptyset\ne T_u\cap T_v\sub T_{u'}\cap  T_{v'}=\{p\}$, we then have  $T_u\cap T_v=\{p\}$.

If $\sgn(p,T_{u'})=\sgn(\widetilde p,S_{u'})=-$, then 
$u=u'1$, because $p\in T_u$. Hence $\widetilde p\in S_{u'1}=S_u$, because the  marked leaf $\widetilde p$ of $S_{u'}$ with $\sgn(\widetilde p,S_{u'})=-$ is passed to the child $S_{u'1}$.   
If $\sgn(p,T_{u'})=\sgn(\widetilde p,S_{u'})=+$, 
 then $u=u'2$ and $\widetilde p\in S_{u'2}=S_u$.  
 
Similarly, if $\sgn(p,T_{v'})=\sgn(\widetilde p,S_{v'})=-$,
then $v=v'1$ and if $\sgn(p,T_{v'})=\sgn(p,S_{v'})=+$, then 
$v=v'2$, because $p\in T_v$.  In both cases, $\widetilde p\in S_v$. 

In each of these cases,  $p$ is a marked leaf in $T_u$ and $T_v$, and $\widetilde p$ is a marked leaf in  $S_u$ and  $S_v$.
 In particular, 
$\{\widetilde p\}\sub S_u\cap S_v\sub S_{u'}\cap S_{v'}=\{\widetilde p\}$ and so 
$S_u\cap S_v=\{\widetilde p\}$.  So both $T_u\cap T_v=\{p\}$ and 
  $S_u\cap S_v=\{\widetilde p\}$ are singleton sets consisting of marked leaves as claimed.
 Since signed marked leaves are passed to children with the same sign, we have 
 $$ \sgn(p,T_{u})= \sgn(p,T_{u'})=\sgn(\widetilde p ,S_{u'})=\sgn(\widetilde p,S_{u}). $$
 Similarly, we conclude that $\sgn(p,T_{v})= \sgn(\widetilde p,S_{v})$.  The statement  follows.  \qed
\end{proof} 

We are now ready to prove 
 Theorem~\ref{t:infinite}, and  Theorem~\ref{criterion} as an immediate consequence.   
 

\begin{proofof} {\em Proof of Theorem~\ref{t:infinite}.} Let $m$ be as in the statement, and consider arbitrary trees $T$ and $S$ in the class $\mc{T}_{m}$.  For each $n\in \N$ we consider 
the decompositions  $T=\bigcup_{u\in W_n} T_{u}$ 
and $S=\bigcup_{u\in W_n} S_{u}$ as defined earlier in this section. 
Here of course, $W_n=W_n(\mathcal{A})$, where $\mathcal{A}=\{1,2,\dots, m\}$. 

We want to show that  decompositions of $T$ and $S$ for different levels $n\in \N$  have the properties 
in Proposition~\ref{nested}. In this proposition the index $i$ for fixed level $n$ corresponds to the words $u\in W_n$. 

The spaces $T$ and $S$ are trees and hence compact. The sets $T_u$ and $S_u$ appearing in their decompositions are subtrees and hence non-empty and compact. Conditions (i), (ii), and (iii) in Proposition~\ref{nested} follow from  
 Lemma~\ref{lem:condtrue}~(iii), \eqref{eq:piecedecomp}, and Lemma~\ref{lem:diamto0}, respectively.  
Finally, \eqref{eq:iff1} and \eqref{eq:iff2} follow
from Lemma~\ref{lem:condtrue2}~(i) and~(ii).  

Proposition~\ref{nested} implies $T$ and $S$ are homeomorphic as desired.  \qed
\end{proofof}

\begin{proofof}{\em Proof of Theorem~\ref{criterion}.} 
 As we have seen in 
Section~\ref{s:CSST}, the CSST $\Te$ is a metric tree with the properties
(i) and (ii) as in the statement (see Proposition~\ref{prop:CSST}
and  Proposition~\ref{prop:T123}). In particular, $\Te$ belongs to the class $\T_3$. 
Since  these properties (i) and (ii) are obviously invariant
under homeomorphisms, every metric tree $T$ homeomorphic to $\Te$ has these properties. 

Conversely, suppose that $T$ is a metric tree with  properties (i) and (ii). Then $T$ belongs to the class $\T_{3}$. So  Theorem~\ref{t:infinite} for  $m=3$ implies that  $T$ and $\Te$ are homeomorphic.
 \qed   
\end{proofof} 

The method of proof for  Theorem~\ref{t:infinite} can be used to a establish a slightly stronger result for $m=3$. 
 
 \begin{theorem}\label{t:norm} Let $T$ and $S$ be trees in 
$\T_3$. Suppose $p_1,p_2,p_3\in T$ are three distinct leaves of  
 $T$, and $q_1,q_2,q_3\in S$  are three distinct leaves of  
 $S$. Then there exists a homeomorphism $f\: T\ra S$ such that 
 $f(p_k)=q_k$ for $k=1,2,3$. \end{theorem}
 
 Note that $-1,1\in \Te$ are leaves of $\Te$ as follows from 
 Lemma~\ref{lem:arc1}~(ii). Moreover, $i\in \Te$ is also a leaf of $\Te$, because the set 
 $$\Te\setminus\{i\}=\Te_1\cup\Te_2\cup(\Te_3\setminus\{i\})=
 \Te_1\cup\Te_2\cup g_3(\Te\setminus\{1\})$$
 is connected. Hence  $\Te$, and so by Theorem~\ref{t:infinite} every tree in $\T_3$, has at least three  leaves (actually infinitely many).
 If we apply Theorem~\ref{t:norm}  to $S=\Te$, then we see that 
  if $T$ is a tree in $\T_3$ with three distinct leaves $p_1$, $p_2$, $p_3$, then there exists a homeomorphism $f\: T\ra \Te$ such that 
  $f(p_1)=-1$, $f(p_2)=1$, and $f(p_3)=i$.

  \begin{proofof}{\em Proof of Theorem~\ref{t:norm}.} 
We will employ a slight modification of our decomposition and coding procedure. The underlying alphabet corresponds to the case $m=3$, 
and so $\mathcal{A}=\{1,2, 3\}$. We describe this  for the tree $T$. Essentially, one wants to use the leaves $p_1$, $p_2$, $p_3$ of $T$ as additional  marked leaves for any of the inductively defined subtrees $T_u$ for  $n\in \N$ and $u\in W_n=W_n(\mathcal{A})$ if it   contains any of these leaves. Here $p_1$   carries the sign $-$, while $p_2$ and $p_3$ carry the sign $+$. 

Instead of starting the decomposition process  with   a branch point $c\in T$ of maximal height, one chooses a branch point $c$ so that 
the leaves $p_1$, $p_2$, $p_3$ lie in distinct branches 
$T_1$, $T_2$, $T_3$ of $c$ in $T$, respectively. To find such a branch point, one  
 travels from $p_1$ along  $[p_1,p_2]$ until one first 
 meets $[p_2,p_3]$ in a point $c$. Then the sets $[p_1, c)$, 
 $[p_2, c)$, $[p_3, c)$ are pairwise disjoint. For $k,\ell\in  \mathcal{A}$ with  $k\ne \ell$ the set $[p_k, c)\cup \{c\}\cup (c,p_\ell]$ is an arc with endpoints $p_k$ and $p_\ell$, and so it must agree with  $[p_k, p_\ell]$. In particular, 
 $c\in [p_k, p_\ell]$.  Since each  point $p_k$ is a leaf, it easily follows from Lemma~\ref{lem:vartree}~(iii) that $c\ne p_1, p_2, p_3$. Indeed, if $c=p_1$ for example, then $c=p_1\in [p_2, p_3]$ and so 
 $p_2$ and $p_3$ would lie in different components of $T\setminus \{p_1\}$. This is impossible, because $p_1$ is a leaf of $T$ and so 
 $T\setminus \{p_1\}$ is connected.   
 
 We conclude that 
 the connected sets  $[p_1, c)$, 
 $[p_2, c)$, $[p_3, c)$ are non-empty and must lie in different branches $T_1$, $T_2$, $T_3$ of $c$. In particular, $c$ is a branch point of $T$. We can choose the labels so that $p_k\in T_k$ for $k=1,2,3$. 
 The point $c$ is a marked leaf in each of theses branches with a  sign chosen  as before. With the additional signs for the distinguished leaves, we then have the set of marked leaves
 $\{p_1^-, c^+\}$ in $T_1$, $\{c^{-}, p_2^+\}$ in $T_2$, and 
 $\{c^-, p_3^+\}$ in $T_3$. 
 
 We now continue inductively 
 as before. If we  have already constructed a subtree 
 $T_u$ for some $n\in \N$ and $u\in W_n$ with one or two signed marked leaves, then we   decompose $T_u$ into three branches labeled 
$T_{u1}$, $T_{u2}$, $T_{u3}$ by using a suitable branch point $c\in T_u$.  Namely, if   $T_u$ is a leaf-tile and has one marked leaf  $a\in T_u$, we choose a branch point $c\in T_u\setminus \{a\}$ with maximal height $H_T(c)$. If 
$T_u$ is an arc-tile with two marked leaves $\{a,b\}\sub T_u$ we choose a branch point $c\in T_u$ of maximal height on $(a,b)\sub T_u$. 

Marked leaves and their signs are assigned to the children 
$T_{u1}$, $T_{u2}$, $T_{u3}$ of $T_u$ as before. In particular, 
 a marked leaf $x^-$ of  $T_u$ is passed to  $T_{u1}$ with the same sign. Similarly, a marked leaf $x^+$ of $T_u$
 is passed to $T_{u2}$ with the same sign. If we continue in this manner, we obtain subtrees $T_u$ with one or two signed marked leaves for all levels $n\in \N$ and $u\in W_n$.

We  apply the same procedure for the tree $S$ and its leaves 
$q_1$, $q_2$, $q_3$. Then Lemma~\ref{lem:condtrue}, Lemma~\ref{lem:diamto0}, and  Lemma~\ref{lem:condtrue2} are true (with almost identical proofs) for the decompositions of $T$ 
and $S$ obtained in this way. 
 The argument  in the proof of Theorem~\ref{t:infinite}   based on 
Proposition~\ref{nested} now guarantees the existence of a  homeomorphism 
$f\: T\ra S$ such that 
\begin{equation}\label{eq:ttot}
f(T_u)=S_u \text  { for all $n\in \N$ and $u\in W_n$.}  
\end{equation}

In our construction $p_1\in T_1$ carries the sign $-$ and is 
hence passed to $T_{11}$ with the same sign; so $p_1\in T_{11}$. Repeating this argument, we see the 
$$p_1\in T_1\cap T_{11}\cap T_{111}\cap \dots. $$ 
The latter nested intersection of compact sets  cannot contain
 more than one point, because by  Lemma~\ref{lem:condtrue2} the diameters of our subtrees $T_u$,
 $u\in W_n$,  approach $0$ uniformly as $n\to \infty$.
  Thus,  $ \{p_1\}= T_1\cap T_{11}\cap T_{111}\cap \dots\,.$ 
 The same argument shows that 
$\{q_1\}=S_1\cap S_{11}\cap S_{111}\cap  \dots\, ,$ 
 and so 
 \eqref{eq:ttot} implies that $f(p_1)=q_1$.

 Similarly, the points $p_2$, $p_3$, $q_2$, $q_3$ carry the sign $+$ in their respective trees. This leads to 
\begin{align*}  \{p_2\}&=T_2\cap T_{22}\cap T_{222}\cap \dots, \quad 
  \{q_2\}=S_2\cap S_{22}\cap S_{222}\cap \dots,  \\
   \{p_3\}&=T_3\cap T_{32}\cap T_{322}\cap \dots,  \quad 
  \{q_3\}=S_3\cap S_{32}\cap S_{322}\cap \dots,  
 \end{align*} 
which by \eqref{eq:ttot}  gives  $f(p_2)=q_2$ and $f(p_3)=q_3$.

 We have shown the existence of a homeomorphism $f\:T\ra S$ with the desired normalization. \qed   
\end{proofof}

\begin{acknowledgement} The authors would like to thank Daniel Meyer for many valuable comments on this paper. The first author was partially supported by NSF grants DMS-1506099 and  DMS-1808856.
\end{acknowledgement}
%


%
%

%
%

\end{document}